\documentclass[11pt,reqno]{amsart}

\usepackage[utf8]{inputenc}

\usepackage{lineno}
\usepackage{amsaddr}
\usepackage{a4wide}
\usepackage{hyperref}

\usepackage[all]{xy}

\usepackage{tikz}
\usetikzlibrary{arrows}

\usepackage{marginnote}

\usepackage{amsmath,amsthm,latexsym,xspace,amscd,amssymb,xy}               
\usepackage{color}
\usepackage{enumerate}

\usepackage[backend=biber,defernumbers=false,firstinits=true, url=false, maxbibnames=9]{biblatex}
\DeclareUnicodeCharacter{2217}{~}

\DeclareMathOperator {\Span}{Span}
\DeclareMathOperator {\Ann}{Ann}
\DeclareMathOperator {\id}{id}

\newtheorem{theorem}{Theorem}[section]
\newtheorem{lemma}[theorem]{Lemma}
\newtheorem{corollary}[theorem]{Corollary}
\newtheorem{proposition}[theorem]{Proposition}

\theoremstyle{definition}
\newtheorem{definition}[theorem]{Definition}
\newtheorem{example}[theorem]{Example}
\newtheorem{remark}[theorem]{Remark}

\author{Daniel Lännström}
\address{Department of Mathematics and Natural Sciences,
Blekinge Institute of Technology,
SE-37179 Karlskrona, Sweden}

\email{{\scriptsize daniel.lannstrom@bth.se}}

\date{\today}

\title{The graded structure of algebraic Cuntz-Pimsner rings}
\begin{document}
\maketitle

\begin{abstract}
Algebraic Cuntz-Pimsner rings are naturally $\mathbb{Z}$-graded rings that generalize corner skew Laurent polynomial rings, Leavitt path algebras and unperforated $\mathbb{Z}$-graded Steinberg algebras. In this article, we  characterize strongly, epsilon-strongly and nearly epsilon-strongly $\mathbb{Z}$-graded algebraic Cuntz-Pimsner rings up to graded isomorphism. We recover two results by Hazrat on when corner skew Laurent polynomial rings and Leavitt path algebras are strongly graded. As a further application, we characterize noetherian and artinian corner skew Laurent polynomial rings.
%
\end{abstract}

\section{Introduction}


The Cuntz-Pimsner $C^*$-algebras were first introduced by Pimsner in \cite{pimsner12class} and further studied by Katsura in \cite{katsura2004c}. The Cuntz-Pimsner algebra is constructed from a $C^*$-correspondence and comes equipped with a natural gauge action. In a recent article, Chirvasitu \cite{2018arXiv180512318C} obtained necessary and sufficient conditions for the gauge action to be free. 
The \emph{(algebraic) Cuntz-Pimsner rings} were introduced by Carlsen and Ortega in \cite{carlsen2011algebraic} as algebraic analogues of the Cuntz-Pimsner algebras, and simplicity of Cuntz-Pimsner rings were studied in \cite{carlsen2012simple}. These  rings are interesting to us since they generalize some  very famous families of rings. Indeed, Carlsen and Ortega originally gave two important examples of rings realizable as Cuntz-Pimsner rings: \emph{Leavitt path algebras}  (see \cite[Expl. 5.8]{carlsen2011algebraic} and Section \ref{sec:lpa}) and \emph{corner skew Laurent polynomial rings} (see \cite[Expl. 5.7]{carlsen2011algebraic} and Section \ref{sec:corner}). Recently, Clark, Fletcher, Hazrat and Li \cite{2018arXiv180810114O} showed that unperforated $\mathbb{Z}$-graded Steinberg algebras are also realizable as  Cuntz-Pimsner rings. The Cuntz-Pimsner rings do not come with a gauge action but instead a natural $\mathbb{Z}$-grading. This grading is the main object of study in this article. 

In the case of Leavitt path algebras, the natural $\mathbb{Z}$-grading was systematically investigated by Hazrat \cite{hazrat2013graded}. In particular, he obtained necessary and sufficient conditions for the Leavitt path algebra of a finite graph to be strongly $\mathbb{Z}$-graded (see \cite[Thm. 3.15]{hazrat2013graded}). The class of \emph{epsilon-strongly graded rings} was first introduced by Nystedt, Öinert and Pinedo in \cite{nystedt2016epsilon} as a generalization of unital strongly graded rings. This subclass of graded rings has been investigated further by the author in \cite{lannstrom2018chain, lannstrom2018induced}. Interestingly, the Leavitt path algebra of a finite graph was proved to be epsilon-strongly $\mathbb{Z}$-graded by Nystedt and Öinert (see \cite[Thm. 1.2]{nystedt2017epsilon}). Seeking to extend their result, they introduced the notion of a \emph{nearly epsilon-strongly graded ring} (see Definition \ref{def:nystedt_epsilon}) and proved that every Leavitt path algebra (even for infinite graphs) is nearly epsilon-strongly $\mathbb{Z}$-graded (see \cite[Thm. 1.3]{nystedt2017epsilon}). In other words, there are sufficient conditions in the literature for the natural $\mathbb{Z}$-grading of a Leavitt path algebra to be strong, epsilon-strong and nearly epsilon-strong respectively. These types of gradings have certain structural properties that help us understand the Leavitt path algebras. The present work began as an effort to generalize the previously mentioned results about Leavitt path algebras to a larger class of Cuntz-Pimsner rings. It turns out that we can obtain partial characterizations of nearly epsilon-strongly and epsilon-strongly graded Cuntz-Pimsner rings (see Theorem \ref{thm:1} and Theorem \ref{thm:epsilon}). For unital strongly graded Cuntz-Pimsner rings we obtain a complete characterization (see Theorem \ref{thm:2}). For that purpose, we obtain sufficient conditions for a Cuntz-Pimsner ring to be strongly graded (see Corollary \ref{cor:cuntz_strongly}). In particular, we recover Hazrat's results on Leavitt path algebras (see Corollary \ref{cor:lpa_strong}) and corner skew Laurent polynomial ring (see Corollary \ref{cor:fractional_strong}) as special cases.

\smallskip

Carlsen and Ortega \cite{carlsen2011algebraic} constructed the Cuntz-Pimsner rings using a categorical approach. Let $R$ be an associative but not necessarily unital ring. Recall (see \cite[Def. 1.1]{carlsen2011algebraic}) that an \emph{$R$-system} is a triple $(P, Q, \psi)$ where $P$ and $Q$ are $R$-bimodules and $\psi \colon P \otimes_R Q \to R$ is an $R$-bimodule homomorphism where $P \otimes_R Q$ denotes the balanced tensor product. A technical assumption called Condition (FS) (see Definintion \ref{def:cond_fs}) is generally imposed on the $R$-system $(P,Q,\psi)$.  We will introduce two special types of $R$-systems called \emph{s-unital} and \emph{unital $R$-systems} (see Definition \ref{def:s-unital}). Given an $R$-system, Carlsen and Ortega considered representations of that system. This is the key definition in their construction:

\begin{definition}(\cite[Def. 1.2, Def. 3.3]{carlsen2011algebraic})
Let $R$ be a ring and let $(P,Q,\psi)$ be an $R$-system. A \emph{covariant representation} is a tuple $(S, T, \sigma, B)$ such that the following assertions hold:
\begin{enumerate}[(a)]
\begin{item}
$B$ is a ring;
\end{item}
\begin{item}
$S \colon P \to B$ and $T \colon Q \to B$ are additive maps;
\end{item}
\begin{item}
$\sigma \colon R \to B$ is a ring homomorphism;
\end{item}
\begin{item}
$S(pr)=S(p)\sigma(r), S(rp)=\sigma(r)S(p), T(qr)=T(q)\sigma(r), T(rq)=\sigma(r)T(q)$ for all $r \in R$, $q \in Q$ and $p \in P$;
\end{item}
\begin{item}
$\sigma(\psi(p \otimes q)) = S(p)T(q)$ for all $p \in P$ and $q \in Q$.
\end{item}
\end{enumerate}
The covariant representation $(S,T,\sigma, B)$ is \emph{injective} if the map $\sigma$ is injective. The covariant representation $(S,T,\sigma, B)$ is \emph{surjective} if $B$ is generated as a ring by $\sigma(R) \cup S(P) \cup T(Q)$. 

A surjective covariant representation $(S,T,\sigma, B)$ is called \emph{graded} if there is a $\mathbb{Z}$-grading $\{ B_i \}_{i \in \mathbb{Z}}$ of $B$ such that $\sigma(R) \subseteq B_0$, $T(Q) \subseteq B_1$ and $S(P) \subseteq B_{-1}$. 
\label{def:covariant_representation}
\end{definition}
\begin{remark}
Let $(S,T,\sigma,B)$ be a covariant representation and assume that $B$ is $\mathbb{Z}$-graded. Note that $(S,T,\sigma,B)$ is a graded covariant representation if and only if the grading of $B$ is compatible with the representation structure. 
\end{remark}

Carlsen and Ortega \cite{carlsen2011algebraic} then considered the category of surjective covariant representations of $(P,Q,\psi)$ denoted by $\mathcal{C}_{(P,Q,\psi)}$. The maps between $(S, T, \sigma, B)$ and $(S', T', \sigma', B')$ are ring homomorphisms $\phi \colon B \to B'$ such that $\phi \circ S = S'$, $\phi \circ T = T'$ and $\phi \circ \sigma = \sigma'$. We write $(S,T,\sigma, B) \cong_{\text{r}} (S', T', \sigma', B')$ if the covariant representations are isomorphic as objects in $\mathcal{C}_{(P,Q,\psi)}$. In the case when $(P,Q,\psi)$ satisfies Condition (FS) (see Definition \ref{def:cond_fs}), they obtained a complete characterization of injective, graded, surjective covariant representations up to isomorphism in $\mathcal{C}_{(P,Q,\psi)}$ (see \cite[Sect. 7]{carlsen2011algebraic}). 
The \emph{Cuntz-Pimsner rings} are  defined as certain universal covariant representations (see Definition \ref{def:cp_ring}). Unlike in the $C^*$-setting, the Cuntz-Pimsner ring is not well-defined for all $R$-systems $(P,Q,\psi)$ (see \cite[Expl. 4.11]{carlsen2011algebraic}).


Let both $R$ and $(P,Q,\psi)$ vary. If a $\mathbb{Z}$-graded ring $B$ shows up in a graded covariant representation $(S,T,\sigma, B)$ of some $R$-system $(P,Q,\psi)$, then we call $B$ a \emph{representation ring}. Following Clark, Fletcher, Hazrat and Li \cite{2018arXiv180810114O}, we then say that $B$ is \emph{realized by} the representation $(S,T,\sigma,B)$ of the $R$-system $(P,Q,\psi)$.  


The key new technique of this article is to consider a special type of graded covariant representations:
\begin{definition}
Let $R$ be a ring, let $(P,Q,\psi)$ be an $R$-system and let $(S,T,\sigma,B)$ be a graded covariant representation of $(P,Q,\psi)$. For $k \geq 0$, let $I_{\psi,\sigma}^{(k)}$ be the $B_0$-ideal generated by the set $\{ \sigma(\psi_k (p \otimes q)) \mid p \in P^{\otimes k}, q \in Q^{\otimes k} \} \subseteq B_0$.  We call $(S,T,\sigma, B)$ a \emph{semi-full} covariant representation if $B_{-k} B_k = I_{\psi,\sigma}^{(k)}$ for every $k \geq 0$. 
\label{def:semi-full}
\end{definition}
\begin{remark}
A $C^*$-correspondence $(A,E,\phi)$ is called \emph{full} if the closure of $\langle x, y \rangle$ for $x,y \in E$ spans $A$. One way to generalize this to the algebraic setting is to require that $\psi$ be surjective. Semi-fullness is a weaker condition. Indeed, if $R$ is unital and $\psi$ is surjective, then every graded covariant representation of $(P,Q,\psi)$ is semi-full.  
\end{remark}

Below is an outline of the rest of this article:

\smallskip

In Section \ref{sec:prelim}, we recall the definitions of nearly epsilon-strongly graded rings and algebraic Cuntz-Pimsner rings. 

In Section \ref{sec:necessary}, we prove that certain nearly epsilon-strongly $\mathbb{Z}$-graded Cuntz-Pimsner rings can be realized from semi-full covariant representations (see Corollary \ref{cor:reduction}). This is based on recent work by Clark, Fletcher, Hazrat and Li \cite{2018arXiv180810114O} and  is the crucial reduction step in the characterization.

In Section \ref{sec:strongly}, we find sufficient conditions for an injective and graded covariant representation to be strongly $\mathbb{Z}$-graded (see Proposition \ref{prop:strong_suff}). Using our general theorems, we recover two results by Hazrat as special cases (see Corollary \ref{cor:lpa_strong} and Corollary \ref{cor:fractional_strong}).

In Section \ref{sec:epsilon}, we obtain sufficient conditions for an injective and semi-full covariant representation ring to be nearly epsilon-strongly $\mathbb{Z}$-graded and epsilon-strongly $\mathbb{Z}$-graded respectively (see Proposition \ref{prop:nearly_epsilon_suff} and Proposition \ref{prop:epsilon_suff}). 

In Section \ref{sec:characterization}, we obtain partial characterizations of nearly epsilon-strongly and epsilon-strongly graded Cuntz-Pimsner rings (see Theorem \ref{thm:1} and Theorem \ref{thm:epsilon}). For unital strongly graded Cuntz-Pimsner rings we obtain a complete characterization (see Theorem \ref{thm:2}).

In Section \ref{sec:ex}, we collect some important examples. Notably, we give an example of a Leavitt path algebra realizable as a Cuntz-Pimsner ring in two different ways (see Example \ref{ex:1}). We also give an example of a trivial Cuntz-Pimsner ring that is not nearly epsilon-strongly $\mathbb{Z}$-graded (see Example \ref{ex:2}).

In Section \ref{sec:app}, we apply our results to characterize noetherian and artinian corner skew Laurent polynomial rings (see Corollary \ref{cor:artinian}).

\section{Preliminaries}
\label{sec:prelim}

All rings are assumed to be associative but not necessarily equipped with a multiplicative identity element. Let $R$ be a ring and let $A \subseteq R$ be a subset. The $R$-ideal generated by $A$ is denoted by $(A)$. Let $_R M$ be a left $R$-module and let $B \subseteq M$ be a subset. The \emph{$R$-linear span of $B$}, denoted by $\Span_R B$, is the $R$-submodule of $_R M$ generated by $B$. More precisely, $\Span_R B = \Big \{ \sum b_i + \sum r_j \cdot b_j \mid b_i, b_j \in B, r_j \in R \Big \},$ where the sums are finite. 

\subsection{Nearly epsilon-strongly graded rings}

Recall that a ring $S$ is called \emph{$\mathbb{Z}$-graded} if there exists a family of additive subsets $\{ S_i \}_{i \in \mathbb{Z}}$ of $S$ such that $S=\bigoplus_{i \in \mathbb{Z}}S_i$ and $S_m S_n \subseteq S_{m+n}$ for all $m, n \in \mathbb{Z}$. If the stronger condition $S_m S_n = S_{m+n}$ holds for all $m,n \in \mathbb{Z}$, then the $\mathbb{Z}$-grading $\{S_i \}_{i \in \mathbb{Z}}$ is called \emph{strong}. The subsets $S_i$ are called the \emph{homogeneous components} of $S$. The \emph{support} of $S$ is defined to be the set $\text{Supp}(S) = \{ i \in \mathbb{Z} \mid S_i \ne \{ 0 \} \}.$ The component $S_0$ is called the \emph{principal component} of 
$S$. It is straightforward to show that $S_0$ is a subring of $S$.  Next, let $S=\bigoplus_{i \in \mathbb{Z}} S_i$ and $T=\bigoplus_{i \in \mathbb{Z}} T_i$ be two $\mathbb{Z}$-graded rings. A ring homomorphism $\phi \colon S \to T$ is called \emph{graded} if $\phi(S_i) \subseteq T_i$ for each $i \in \mathbb{Z}$. If $\phi \colon S \xrightarrow{\sim} T$ is a graded ring isomorphism, then we write $S \cong_{\text{gr}} T$ and say that $S$ and $T$ are \emph{graded isomorphic}. 

Let $R$ be a ring. Recall that a left (right) $R$-module $_R M$ is called \emph{left (right) s-unital} if for every $x \in M$ there exists some $r_x \in R$ such that $r_x \cdot x = x$ ($x \cdot r_x = x$). A left (right) $R$-module $_R M$ is called \emph{left (right) unital} if there exists some $r \in R$ such that $r \cdot x = x$ ($x \cdot r = x$) for every $x \in M$. Let $R, S$ be rings. A bimodule $_R M _S$ is called \emph{s-unital} (\emph{unital}) if $_ R M$ is left s-unital (unital) and $M _S$ is right s-unital (unital). In particular, an ideal $I$ of $R$ is called \emph{s-unital} (\emph{unital}) if $_R I_R$ is s-unital (unital). 
\begin{remark}
Let $R$ be a ring. It follows from \cite[Thm. 1]{tominaga1976s} that if $M$ is a left (right) s-unital $R$-module, then for any positive integer $n$ and elements $x_1, x_2, \dots, x_n \in M$ there exists some $r \in R$ such that $r \cdot x_i = x_i$ ($x_i \cdot r = x_i$) for all $i \in \{ 1, \dots, n \}$.
\label{rem:s-unital}
\end{remark}

If $S$ is a $\mathbb{Z}$-graded ring, then $S_i$ is an $S_0$-bimodule for every $i \in \mathbb{Z}$ (see \cite[Rmk. 1.1.2]{nastasescu2004methods}). Note that $S_i S_{-i}$ is an ideal of $S_0$ for every $i \in \mathbb{Z}$. Hence, in particular, $S_i$ is an $S_{i} S_{-i} \text{--} S_{-i} S_i$-bimodule for each $i \in \mathbb{Z}$.  The following definitions were introduced by Nystedt and Öinert:

\begin{definition}(\cite[Def. 3.1, Def. 3.2, Def. 3.3]{nystedt2017epsilon})
Let $S=\bigoplus_{i \in \mathbb{Z}} S_i$ be a $\mathbb{Z}$-graded ring. 
\begin{enumerate}[(a)]
\begin{item}
If $S_i$ is an s-unital $S_i S_{-i} \text{--} S_{-i} S_i$-bimodule for each $i \in \mathbb{Z}$, then $S$ is called \emph{nearly epsilon-strongly} $\mathbb{Z}$-graded.
\end{item}
\begin{item}
If $S_i$ is a unital $S_i S_{-i} \text{--} S_{-i} S_i$-bimodule for each $i \in \mathbb{Z}$, then $S$ is called \emph{epsilon-strongly} $\mathbb{Z}$-graded.
\end{item}
\begin{item}
(cf. \cite[Def. 4.5]{clark2018generalized}) If $S_i = S_i S_{-i} S_i$ for every $i \in \mathbb{Z}$, then $S$ is called \emph{symmetrically} $\mathbb{Z}$-graded.
\end{item}
\end{enumerate}
\label{def:nystedt_epsilon}
\end{definition}
\begin{remark}We make two remarks regarding Definition \ref{def:nystedt_epsilon}.
\begin{enumerate}[(a)]
\begin{item}
Nystedt and Öinert made these definitions for general group graded rings graded by an arbitrary group. However, in this article we will only consider the special case of $\mathbb{Z}$-graded rings.
\end{item}
\begin{item}
If $S$ is epsilon-strongly $\mathbb{Z}$-graded, then $S$ is a unital ring (see \cite[Prop. 3.8]{lannstrom2018induced}). In  other words, only unital rings admit an epsilon-strong grading. 
\end{item}
\end{enumerate}
\label{rem:unital_epsilon}
\end{remark}

We recall the following characterizations of nearly epsilon-strongly graded rings and epsilon-strongly graded rings.

\begin{proposition}
(\cite[Prop. 3.1, Prop. 3.3]{nystedt2017epsilon})
Let $S=\bigoplus_{i \in \mathbb{Z}} S_i$ be a $\mathbb{Z}$-graded ring. The following assertions hold:
\begin{enumerate}[(a)]
\begin{item}
$S$ is nearly epsilon-strongly $\mathbb{Z}$-graded if and only if $S$ is symmetrically $\mathbb{Z}$-graded and $S_i S_{-i}$ is an s-unital ideal for each $i \in \mathbb{Z}$;
\end{item}
\begin{item}
$S$ is epsilon-strongly $\mathbb{Z}$-graded if and only if $S$ is symmetrically $\mathbb{Z}$-graded and $S_i S_{-i}$ is a unital ideal for each $i \in \mathbb{Z}$.
\end{item}
\end{enumerate}
\label{prop:nearly_char}
\end{proposition}

Moreover, the following implications hold (see \cite[Rem. 3.4(a)]{lannstrom2018induced}):
\begin{equation}
\text{unital strongly graded} \Rightarrow \text{ epsilon strongly graded} \Rightarrow \text{nearly epsilon-strongly graded}.
\label{eq:implications}
\end{equation}

\subsection{The Toeplitz representation}

Let $(P,Q,\psi)$ be an $R$-system. Put $P^{\otimes 0} = Q^{\otimes 0}=R$ and $\psi_0(r_1 \otimes r_2) = r_1 r_2$. Let $\psi_1 = \psi$. For $n > 1$, recursively define $Q^{\otimes n} = Q^{\otimes {n-1}} \otimes Q$ and $P^{\otimes n} = P \otimes P^{\otimes {n-1}}$. Let $\psi_n \colon P^{\otimes n} \otimes Q^{\otimes n} \to R$ be defined by, $$\psi_n((p_1 \otimes p_2) \otimes (q_2 \otimes q_1)) = \psi(p_1 \cdot \psi_{n-1}(p_2 \otimes q_2), q_1),$$ for $p_1 \in P, p_2 \in P^{\otimes {n-1}}, q_1 \in Q,$ and $q_2 \in Q^{\otimes {n-1}}.$ Then, $(P^{\otimes n}, Q^{\otimes n}, \psi_n)$ is an $R$-system for each $n \geq 0$. Furthermore, by \cite[Lem. 1.5]{carlsen2011algebraic}, if $(S,T,\sigma,B)$ is a covariant representation of $(P,Q,\psi)$, then $(S^n, T^n, \sigma, B)$ is a covariant representation of $(P^{\otimes n}, Q^{\otimes n}, \psi_n)$ where $S^n \colon P^{\otimes n} \to B$ and $T^n \colon Q^{\otimes n} \to B$ are maps satisfying the equations $S^n(p_1 \otimes \dots \otimes p_n) = S(p_1)S(p_2)\dots S(p_n)$ and $T^n(q_1 \otimes \dots \otimes q_n) = T(q_1) T(q_2) \dots T(q_n)$ for $q_i \in Q$ and $p_j \in P$.

%

Carlsen and Ortega proved (see \cite[Thm. 1.7]{carlsen2011algebraic}) that there is an injective, surjective and graded covariant representation that satisfies a universal property. This covariant representation is called the \emph{Toeplitz representation} and is denoted by $(\iota_Q, \iota_P, \iota_R, \mathcal{T}_{(P,Q,\psi)})$. The ring $\mathcal{T}_{(P,Q,\psi)}$ is called the \emph{Toeplitz ring}. We recall (see \cite[Thm. 1.7, Prop. 3.1]{carlsen2011algebraic}) the canonical $\mathbb{Z}$-grading of the Toeplitz ring. The ring homomorphism $\iota_R \colon R \to \mathcal{T}_{(P,Q,\psi)}$ (cf. Definition \ref{def:covariant_representation}(c)), turns the ring $\mathcal{T}_{(P,Q,\psi)}$ into an $R$-algebra. For every pair $(m,n)$ of  non-negative integers, consider the following additive subset of $\mathcal{T}_{(P,Q,\psi)}$, 
\begin{align*}
 \mathcal{T}_{(m,n)} &=  \Span_R \{ \iota_{Q^{\otimes m}}(q) \iota_{P^{\otimes n}}(p) \mid  q \in Q^{\otimes m}, p \in P^{\otimes n} \}. 
 \end{align*} 
Carlsen and Ortega showed that $\mathcal{T}_{(P,Q,\psi)} = \bigoplus_{m,n \geq 0} \mathcal{T}_{(m,n)}$ is a semigroup grading of $\mathcal{T}_{(P,Q,\psi)}$ (see \cite[Def. 1.6]{carlsen2011algebraic}). 
For every $i \in \mathbb{Z},$ define,
\begin{equation}
\mathcal{T}_i = \bigoplus_{\substack{i \in \mathbb{Z} \\ m-n=i}} \mathcal{T}_{(m,n)}.
\label{eq:grading}
\end{equation}
The canonical $\mathbb{Z}$-grading of the Toeplitz ring is then given by $\mathcal{T}_{(P,Q,\psi)} = \bigoplus_{i \in \mathbb{Z}} \mathcal{T}_i$. Moreover, the Toeplitz ring satisfies the following universal property:

\begin{theorem}(\cite[Thm. 1.7, Prop. 3.2]{carlsen2011algebraic})
Let $R$ be a ring and let $(P,Q,\psi)$ be an $R$-system. Let $\mathcal{T}_{(P,Q,\psi)} = \bigoplus_{i \in \mathbb{Z}} \mathcal{T}_i$ be the Toeplitz ring associated to $(P,Q,\psi)$ and let $(S,T,\sigma, B)$ be any graded covariant representation of $(P,Q,\psi)$. Then there is a unique $\mathbb{Z}$-graded ring epimorphism $\eta \colon \mathcal{T}_{(P,Q,\psi)} \to B$ such that $\eta \circ \iota_R = \sigma, \eta \circ \iota_Q = T,$ and $\eta \circ \iota_P = S$.
\label{thm:universal}
\end{theorem}

We relate morphisms in the category of graded covariant representations to morphisms in the category of $\mathbb{Z}$-graded rings:



\begin{lemma}
Let $R$ be a ring and let $(P,Q,\psi)$ be an $R$-system. Suppose that $(S,T,\sigma,B)$ and $(S', T', \sigma', B')$ are two graded covariant representations of $(P,Q,\psi)$. If $$\phi \colon (S,T,\sigma,B) \to (S', T', \sigma', B')$$ is a morphism in the category $\mathcal{C}_{(P,Q,\psi)}$ (see the introduction), then $\phi \colon B \to B'$ is a $\mathbb{Z}$-graded ring homomorphism. 
\label{lem:rep_maps}
\end{lemma}
\begin{proof}
Applying Theorem \ref{thm:universal} to $(S,T,\sigma,B)$, it follows that $B_i = \eta(\mathcal{T}_i)$ and hence, by (\ref{eq:grading}), 
\begin{align*}
B_i = \Span_R \{ T(q) S(p) \mid q \in Q^{\otimes m}, p \in P^{\otimes n} \text{ where }m-n=i \},
\end{align*}
for every $i \in \mathbb{Z}$. Similarly,
$B_i' = \Span_R \{ T'(q) S'(p) \mid q \in Q^{\otimes m}, p \in P^{\otimes n} \text{ where }m-n=i \},$
for every $i \in \mathbb{Z}$. Since $\phi \circ T = T'$ and $\phi \circ S = S'$ it follows that $\phi(B_i) \subseteq B_i'$. Thus, $\phi$ is a $\mathbb{Z}$-graded ring homomorphism.
\end{proof}

The following corollary is straightforward to prove:
\begin{corollary}
Let $R$ be a ring and let $(P,Q,\psi)$ be an $R$-system. Suppose that $(S,T,\sigma,B) \cong_{\text{r}} (S', T', \sigma', B')$ are two isomorphic graded covariant representations of $(P,Q,\psi)$. Then, we have that $B \cong_{\text{gr}} B'$.
\label{cor:rep_iso}
\end{corollary}
\subsection{Adjointable operators, Condition (FS) and Cuntz-Pimsner representations}

Recall from the $C^*$-setting, that finite generation of the Hilbert module $E$ is equivalent to the ring of compact operators $B(E)=K(E)$ being unital. In the algebraic setting, the ring of compact operators $K(E)$ is replaced by $\mathcal{F}_P(Q)$ and $\mathcal{F}_Q(P)$ (see \cite[Def. 2.1]{carlsen2011algebraic}). We will later see that if $P,Q$ are finitely generated, then $\mathcal{F}_P(Q)$ and $\mathcal{F}_Q(P)$ are unital (see Proposition \ref{prop:fsprime_char}). For now, we recall the definition of these rings. A right $R$-module homomorphism $T \colon Q_R \to Q_R$ is called \emph{adjointable} if there exists a left $R$-module homomorphism $S \colon _R P \to _R P$ such that $\psi(p \otimes T(q)) = \psi(S(p) \otimes q)$ for all $q \in Q$ and $p \in P$. The set of adjointable homomorphisms is denoted by $\mathcal{L}_P(Q)$ and $\mathcal{L}_Q(P)$. Note that $\mathcal{L}_P(Q)$ and $\mathcal{L}_Q(P)$ are subrings of $\text{End}(Q_R)$ and $\text{End}(_R P)$ respectively.
Given fixed elements $q \in Q$ and $p \in P$, define $\theta_{q,p} \colon Q_R \to Q_R$ and $\theta_{p,q} \colon _R P \to  {}_R P$ by $\theta_{q,p}(x) = q \cdot \psi(p \otimes x)$ and $ \theta_{p,q}(y) =\psi(y \otimes q) \cdot p$ for $x \in Q$ and $y \in P$ respectively. The $R$-linear span of the homomorphisms $\{ \theta_{q,p} \mid q \in Q, p \in P \}$ is denoted by $\mathcal{F}_P(Q)$. Similarly, the $R$-linear span of $\{ \theta_{p,q} \mid q \in Q, p \in P \}$ is denoted by $\mathcal{F}_Q(P)$. It can be proved that $\mathcal{F}_P(Q)$ and $\mathcal{F}_Q(P)$ are two-sided ideals of $\mathcal{L}_P(Q)$ and $\mathcal{L}_Q(P)$ respectively (see \cite[Lem. 2.3]{carlsen2011algebraic}).
The following technical condition was introduced by Carlsen and Ortega:
\begin{definition}
(\cite[Def. 3.4]{carlsen2011algebraic})
Let $R$ be a ring. An $R$-system $(P,Q,\psi)$ is said to satisfy \emph{Condition (FS)}  if for all finite sets $\{ q_1, q_2, \dots, q_n \} \subseteq Q$ and $\{ p_1, p_2, \dots, p_m \} \subseteq P$ there exist some $\Theta \in \mathcal{F}_P(Q)$ and $\Phi \in \mathcal{F}_Q(P)$ such that $\Theta(q_i) = q_i$ and $\Phi(p_j)=p_j$ for all $1 \leq i \leq n$ and $1 \leq j \leq m$.
\label{def:cond_fs}
\end{definition}

Note that we have the following inclusion of rings:
\begin{align}
\mathcal{F}_P(Q) & \subseteq \mathcal{L}_P(Q) \subseteq \text{End}(Q_R), \nonumber \\ 
\mathcal{F}_Q(P) & \subseteq \mathcal{L}_Q(P) \subseteq \text{End}(_R P).
\label{eq:inc1}
\end{align}

Carlsen and Ortega (see \cite[Def. 3.10]{carlsen2011algebraic}) defined maps $\Delta \colon R \to \mathcal{L}_P(Q)$ and $\Gamma \colon R \to \mathcal{L}_Q(P)$ by $\Delta(r)(q) = r q$ and $\Gamma(r)(p) = p r$ for all $r \in R, q \in Q, p \in P$. 

In the $C^*$-setting, it turns out that there are always injective morphisms $\pi_n \colon K(E^{\otimes n}) \to \mathcal{T}_E$ for each $n > 0$. In the algebraic setting, Carlsen and Ortega obtained something similar under the assumption that the system satisfies Condition (FS). Another way to put it is that if the $R$-system satisfies Condition (FS), then there are induced representations of $\mathcal{F}_P(Q)$ and $\mathcal{F}_Q(P)$. Recall that the opposite ring $R^{\text{op}}$ of a ring $R$ has the same additive structure but with a new multiplication defined by $a \star b = ba$ for all $a,b \in R$. 
\begin{proposition}
(\cite[Prop. 3.11]{carlsen2011algebraic})
Let $R$ be a ring, let $(P,Q,\psi)$ be an $R$-system satisfying Condition (FS) and let $(S,T,\sigma,B)$ be a covariant representation of $(P,Q,\psi)$. Then there exist unique ring homomorphisms $\pi_{T, S} \colon \mathcal{F}_P(Q) \to B$  and $\chi_{T, S} \colon \mathcal{F}_Q(P) \to B^{\text{op}}$ such that $\pi_{T, S}(\theta_{q,p}) = T(q) S(p)$ and $\chi_{T,S}(\theta_{p,q}) = S(p) \star T(q)$ for all $q \in Q, p \in P$. The maps satisfy the following equations for all $\Theta \in \mathcal{F}_P(Q)$ and $\Phi \in \mathcal{F}_Q(P)$:
\begin{align}
\pi_{T, S}(\Delta(r) \Theta) = \sigma(r) \pi_{T, S}(\Theta), \qquad & \pi_{T, S}(\Theta \Delta(r)) = \pi_{T, S}(\Theta) \sigma(r) \nonumber \\
\chi_{T, S}(\Gamma(r)\Phi) = \sigma(r) \star \chi_{T, S}(\Phi) , \qquad  & \chi_{T, S}(\Phi \Gamma(r)) = \chi_{T, S}(\Phi) \star \sigma(r) \nonumber \\
\pi_{T, S}(\Theta)T(q) = T(\Theta(q)), \qquad  & \chi_{T, S}(\Phi) \star S(p) = S(\Phi(p)). \label{eq:2.9a} 
\end{align}
Moreover, $\pi_{T,S}(\mathcal{F}_P(Q)) = \chi_{T, S}(\mathcal{F}_Q(P)) = \Span_R \{ T(q) S(p) \mid q \in Q, p \in P \} \subseteq B$.
If $\sigma$ is injective, then the maps $\pi_{T,S}$ and $\chi_{T,S}$ are also injective.
\label{prop:pi}
\end{proposition}

\begin{remark}We make two remarks regarding Proposition \ref{prop:pi}.
\begin{enumerate}[(a)]
\begin{item}
The equation $\chi_{T, S}(\Phi) \star S(p) = S(p) \chi_{T, S}(\Phi) = S(\Phi(p))$ is misprinted in \cite[Prop. 3.11]{carlsen2011algebraic}.
\end{item}
\begin{item}
Following Carlsen and Ortega, let $\pi$ denote the map $\bigcup_m \mathcal{F}_{P^{\otimes m}}(Q^{\otimes m}) \to \mathcal{T}_{(P,Q,\psi)}$ .
\end{item}
\end{enumerate}

\end{remark}

We now recall the definition of the Cuntz-Pimsner invariant representations. If the $R$-system $(P,Q,\psi)$ satisfies Condition (FS), then the Cuntz-Pimsner invariant representations exhaust all injective, surjective graded covariant representations of $(P,Q,\psi)$ up to isomorphism in $\mathcal{C}_{(P,Q,\psi)}$ (see \cite[Rem. 3.30]{carlsen2011algebraic}).

\begin{definition}(\cite[Def. 3.15, Def. 3.16]{carlsen2011algebraic})
Let $R$ be a ring and let $(P,Q,\psi)$ be an $R$-system satisfying Condition (FS). Let $J$ be an ideal of $R$. If $J \subseteq \Delta^{-1}(\mathcal{F}_P(Q))$, then the ideal $J$ is called \emph{$\psi$-compatible}.  If $\ker \Delta \cap J = \{0 \}$, then $J$ is called \emph{faithful}.
For a $\psi$-compatible ideal $J \subseteq R$, let $\mathcal{T}(J)$ be the ideal of $\mathcal{T}_{(P,Q,\psi)}$ generated by the set $\{ \iota_R(x) - \pi(\Delta(x)) \mid x \in J \}$. 
The \emph{Cuntz-Pimsner ring relative to $J$} is defined as the quotient ring $\mathcal{O}_{(P,Q,\psi)} = \mathcal{T}_{(P,Q,\psi)} / \mathcal{T}(J)$. Let $\rho \colon \mathcal{T}_{(P,Q,\psi)} \to \mathcal{O}_{(P,Q,\psi)}$ be the quotient map. Let $\iota_Q^J = \rho \circ \iota_Q$, $\iota_P^J = \rho \circ \iota_P$ and $\iota_R^J = \rho \circ \iota_R$. The covariant representation $(\iota_Q^J, \iota_P^J, \iota_R^J, \mathcal{O}_{(P,Q,\psi)}(J))$ is called the \emph{Cuntz-Pimsner representation relative to $J$}.
\end{definition}

A covariant representation $(S,T,\sigma, B)$ is called \emph{invariant relative to $J$} if $\pi_{T,S}(\Delta(x)) = \sigma(x)$ holds in $B$ for each $x \in J$. The relative Cuntz-Pimsner representation $(\iota_Q^J, \iota_P^J, \iota_R^J, \mathcal{O}_{(P,Q,\psi)}(J))$ is invariant relative to $J$ and satisfies a universal property among invariant representations (see \cite[Thm. 3.18]{carlsen2011algebraic}). Finally, we recall the definition of the Cuntz-Pismner ring:

\begin{definition}(\cite[Def. 5.1]{carlsen2011algebraic})
Let $R$ be a ring and let $(P,Q,\psi)$ be an $R$-system. Suppose that there exists a unique maximal $\psi$-compatible, faithful ideal $J$ of $R$. The \emph{Cuntz-Pimsner ring} is defined as $\mathcal{O}_{(P,Q,\psi)} = \mathcal{O}_{(P,Q,\psi)}(J)=\mathcal{T}_{(P,Q,\psi)}/\mathcal{T}(J)$ and the \emph{Cuntz-Pimsner representation} $(\iota_Q^{CP}, \iota_P^{CP}, \iota_R^{CP}, \mathcal{O}_{(P,Q,\psi)})$ is defined to be $(\iota_Q^J, \iota_P^J, \iota_R^J, \mathcal{O}_{(P,Q,\psi)}(J))$.
\label{def:cp_ring}
\end{definition}

\subsection{Leavitt path algebras}
\label{sec:lpa}
The Leavitt path algebra associated to a directed graph was introduced by Ara, Moreno and Pardo \cite{ara2007nonstable} and by Abrams and Aranda Pino \cite{abrams2005leavitt}.
For a thorough account of the theory of Leavitt path algebras, we refer the reader to the monograph by Abrams, Ara, and Siles Molina \cite{abrams2017leavitt}.
We now recall the realization of Leavitt path algebras as Cuntz-Pimsner rings given by Carlsen and Ortega (see \cite[Expl. 1.10, Expl. 5.9]{carlsen2011algebraic}). They only considered Leavitt path algebras with coefficients in a commutative unital ring, but their construction also works for non-commutative unital rings. Let $K$ be a unital ring that will serve as the coefficient ring. Let $E=(E^0, E^1, s, r)$ be a directed graph consisting of a vertex set $E^0$, an edge set $E^1$ and maps $s \colon E^1 \to E^0$ and $r \colon E^1 \to E^0$ specifying the source vertex $s(f)$ and range vertex $r(f)$ for each edge $f \in E^1$.  For vertices $u, v \in E^0$, let $\delta_{u,v}=1$ if $u=v$ and $\delta_{u,v} = 0$ if $u \ne v$. Moreover, let $\{ \eta_v \mid v \in E^0 \}$ be a copy of the set $E^0$ and similarly let $\{ \eta_f \mid f \in E^1 \}$ and $\{ \eta_{f^*} \mid f \in E^1 \}$ be copies of the set $E^1$. 

\begin{enumerate}[(a)]
\begin{item}
Put $R := \bigoplus_{v \in E^0} K \eta_v$. Define a multiplication on $R$ by $K$-linearly extending the rules $\eta_u \eta_v = \delta_{u,v} \eta_v$ for all $u, v \in E^0$.
\end{item}
\begin{item}
Put $Q := \bigoplus_{f \in E^1} K \eta_f$. Let $R$ act on the left of $Q$ by $K$-linearly extending the rules $\eta_v \cdot \eta_f = \delta_{v, s(f)} \eta_f$ for all $v \in E^0, f \in E^1$. Let $R$ act on the right of $Q$ by $K$-linearly extending the rules $\eta_f \cdot \eta_v = \delta_{v, r(f)} \eta_f$. 
\end{item}
\begin{item}
Put $P := \bigoplus_{f \in E^1} K \eta_{f^*}$. Let $R$ act on the left of $P$ by $K$-linearly extending the rules $\eta_v \cdot \eta_{f^*} = \delta_{v, r(f)} \eta_{f^*}$ for all $v \in E^0, f \in E^1$. Let $R$ act on the right of $P$ by $K$-linearly extending the rules $\eta_{f^*} \cdot \eta_v = \delta_{v, s(f)} \eta_{f^*}$ for all $v \in E^0, f \in E^1$.
\end{item}
\begin{item}
Define an $R$-bimodule homomorphism $\psi \colon P \otimes_R Q \to R$ by $\eta_{f^*} \otimes \eta_{f'} \mapsto \delta_{f, f'} \eta_{r(f)}$ for all $f, f' \in E^1$.
\end{item}
\end{enumerate}
We will refer to the above $R$-system $(P,Q,\psi)$ as the \emph{standard Leavitt path system} associated to the directed graph $E$ (with coefficients in $K$). Carlsen and Ortega proved (see \cite[Expl. 5.8]{carlsen2011algebraic}) that $(P,Q,\psi)$ satisfies Condition (FS), that the Cuntz-Pimsner ring is well-defined and that $\mathcal{O}_{(P,Q,\psi)} \cong_{\text{gr}} L_K(E)$. The covariant representation $(\iota_{Q}^{CP}, \iota_P^{CP}, \iota_R^{CP}, \mathcal{O}_{(P,Q,\psi)})$ is called the \emph{standard Leavitt path algebra covariant representation}.  Clark, Fletcher, Hazrat and Li also obtained these facts using more general methods (see  \cite[Expl. 3.6]{2018arXiv180810114O}).

\subsection{Corner skew Laurent polynomial rings}
\label{sec:corner}
The general construction of fractional skew monoid rings was introduced by Ara, Gonzalez-Barroso, Goodearl and Pardo in \cite{ara2004fractional} as algebraic analogues of certain $C^*$-algebras introduced by Paschke \cite{paschke1980crossed}. Here, we consider the special case of a fractional skew monoid ring by a corner isomorphism which is also called a \emph{corner skew Laurent polynomial ring}.
Let $R$ be a unital ring and let $\alpha \colon R \to eRe$ be a corner ring isomorphism where $e$ is an idempotent of $R$. The corner skew Laurent polynomial ring $R[t_{+}, t_{-}; \alpha]$ is defined to be the universal unital ring satisfying the following conditions:
\begin{enumerate}[(a)]
\begin{item}
There is a unital ring homomorphism $i \colon R \to R[t_{+}, t_{-}; \alpha]$;
\end{item}
\begin{item}
$R[t_{+}, t_{-}; \alpha]$ is the $R$-algebra  satisfying the following equations for every $r \in R$:
\begin{equation*}
t_{-}t_{+} = 1, \qquad t_{+} t_{-} = i(e), \qquad i(r)t_{-} = t_{-} i(\alpha(r)), \qquad t_{+}i(r) =i(\alpha(r)) t_{+}.
\end{equation*}
\end{item}
\end{enumerate}
Moreover, $R[t_{+}, t_{-}; \alpha]$ is $\mathbb{Z}$-graded with $A_0= R$, $A_i = R t_{+}^i$ for $i < 0$ and $A_i = t_{-}^i R$ for $i > 0$. Note that $t_{-} \in A_1$ and $t_{+} \in A_{-1}$! Carlsen and Ortega \cite[Expl. 5.7]{carlsen2011algebraic} proved that the corner skew Laurent polynomial ring $R[t_{+}, t_{-}; \alpha]$ can be realized as a Cuntz-Pimsner ring.

\section{Nearly epsilon-strongly $\mathbb{Z}$-graded rings as Cuntz-Pimsner rings}
\label{sec:necessary}

In this section, we will see that a recent result by Clark, Fletcher, Hazrat and Li \cite{2018arXiv180810114O} will allow us to derive necessary conditions for certain Cuntz-Pimsner rings to be nearly epsilon-strongly $\mathbb{Z}$-graded. Inspired by Exel we make the following definition:

\begin{definition}
(cf. \cite[Def. 4.9]{exel2017partial})
Let $A=\bigoplus_{i \in \mathbb{Z}} A_i$ be a $\mathbb{Z}$-graded ring. If $A_n = (A_1)^n$ and $A_{-n} = (A_{-1})^n$ for $n > 0$, then $A$ is called \emph{semi-saturated}. 
\end{definition}

We show that the Toeplitz ring and any graded covariant representation is semi-saturated.

\begin{proposition}
Let $R$ be a ring and let $(P, Q, \psi)$ be an $R$-system. 
\begin{enumerate}[(a)]
\begin{item}
The Toeplitz ring $\mathcal{T}_{(P,Q,\psi)}=\bigoplus_{i \in \mathbb{Z}} \mathcal{T}_i$ is semi-saturated.
\end{item}
\begin{item}
Let $(S,T,\sigma, B)$ be any graded covariant representation of $(P,Q,\psi)$. Then $B=\bigoplus_{i \in \mathbb{Z}} B_i$ is semi-saturated.
\end{item}
\end{enumerate}
\label{prop:toeplitz_exponent}
\end{proposition}
\begin{proof}
(a): Take an arbitrary integer $t > 0$. It follows from the $\mathbb{Z}$-grading that $(\mathcal{T}_1)^t \subseteq \mathcal{T}_t$. We prove the reverse inclusion. Let $\iota_{Q^{\otimes m}}(q) \iota_{P^{\otimes n}}(p) \in \mathcal{T}_t$ where $q \in Q^{\otimes m}, p \in P^{\otimes n}$ and $m-n=t$. We need to show that $\iota_{Q^{\otimes m}}(q) \iota_{P^{\otimes n}}(p)  \in (\mathcal{T}_1)^t$. 
Suppose $q = f_1 \otimes f_2 \otimes \dots \otimes f_{n+t}$ and $p = g_1 \otimes g_2 \otimes \dots \otimes g_{n}$.
Then, 
\begin{align*}
\iota_{Q^{\otimes m}}(q) \iota_{P^{\otimes n}}(p) = \iota_Q(f_1) \iota_Q(f_2) \dots \iota_Q(f_{t-1}) \iota_{Q^{\otimes (n+1)}}(f_{t} \otimes f_{t+1} \otimes f_t \otimes \dots \otimes f_{n+t}) \iota_{P^{\otimes n}}(p),
\end{align*}
is contained in  $ (\mathcal{T}_1)^t$. Hence, $\mathcal{T}_t = (\mathcal{T}_1)^t$ for $t > 0$. A similar argument shows that $\mathcal{T}_{-t} = (\mathcal{T}_{-1})^{t}$ for $t > 0$.

(b): By Theorem \ref{thm:universal}, there is a $\mathbb{Z}$-graded ring epimorphism $\eta \colon \mathcal{T}_{(P,Q,\psi)} \to B$. Hence, $B_n = \eta(\mathcal{T}_n) = \eta((\mathcal{T}_1)^n) = \eta(\mathcal{T}_1)^n = (B_1)^n$ for any $n > 0$. Similarly, $B_{-n} = (B_{-1})^n$ for any $n > 0$.
\end{proof}

If $M$ is a left $R$-module, then the left annihilator $\Ann_{R}(M) = \{ r \in R \mid r \cdot m = 0 \enspace \forall m \in M \}$ is an ideal of $R$.  If $J$ is an ideal of $R$, then $J^\bot = \{ r \in R \mid rx = xr = 0 \enspace \forall x \in J \}.$ The following result was recently obtained by Clark, Fletcher, Hazrat and Li. Their formulation of the theorem is weaker but they in fact prove the stronger statement below.


\begin{theorem}(\cite[Cor. 3.2]{2018arXiv180810114O})
Let $A=\bigoplus_{i \in \mathbb{Z}} A_i$ be a $\mathbb{Z}$-graded ring satisfying the following assertions:
\begin{enumerate}[(a)]
\begin{item}
$A$ is semi-saturated;
\end{item}
\begin{item}
For $\{a_1, a_2, \dots, a_n \} \subseteq A_1$ there is $r \in A_1 A_{-1}$ such that $r a_l = a_l$ for each $1 \leq l \leq n$, and 
for $\{ b_1, b_2, \dots, b_m \} \subseteq A_{-1}$ there is $s \in A_1 A_{-1}$ such that $b_l s = b_l$ for each $1 \leq l \leq m$;
\end{item}
\begin{item}
$\Ann_{A_0}(A_1) \cap  \Ann_{A_0}(A_1)^\bot = \{ 0 \}$.
\end{item}
\end{enumerate}
 Let $\psi \colon A_{-1} \otimes A_1 \to A_0$ be defined by $\psi(a' \otimes a) = a' a$. Then the $A_0$-system $(A_{-1}, A_1, \psi)$ satisfies Condition (FS). Let $i_{A_{-1}} \colon A_{-1} \to A$, $i_{A_1} \colon A_1 \to A$, $ i_{A_0} \colon A_0 \to A$ denote the inclusion maps and let $J = A_1 A_{-1}$. Then $(i_{A_{-1}}, i_{A_{1}}, i_{A_0}, A)$ is a surjective covariant representation of $(A_{-1}, A_1, \psi)$ and,
\begin{equation}
(i_{A_{-1}}, i_{A_1}, i_{A_0}, A) \cong_{\text{r}} (\iota_{A_{-1}}^J, \iota_{A_1}^J, \iota_{A_0}^J, \mathcal{O}_{(A_{-1},A_1, \psi')}(J)).
\label{eq:rep_iso}
\end{equation}
Furthermore, $J$ is faithfully maximal, hence, $$(\iota_{A_{-1}}^J, \iota_{A_1}^J, \iota_{A_0}^J, \mathcal{O}_{(A_{-1},A_1, \psi')}(J)) = (\iota_{A_{-1}}^{CP}, \iota_{A_1}^{CP}, \iota_{A_0}^{CP}, \mathcal{O}_{(A_{-1},A_1, \psi')}).$$
\noindent
In particular, we have that $A \cong_{\text{gr}} \mathcal{O}_{(A_{-1}, A_1, \psi')}$. 

\label{thm:clark}
\end{theorem}
\begin{proof}
Note that $(A_{-1}, A_1, \psi)$ is an $A_0$-system. Since $A$ is semi-saturated, it follows that $A$ is generated as a ring by $A_{-1} \cup A_{1} \cup A_0$. Hence, $(i_{A_1}, i_{A_{-1}}, i_{A_0}, A)$ is a surjective covariant representation. In the proof of \cite[Thm. 3.1]{2018arXiv180810114O}, they show that $(A_{-1}, A_1, \psi)$ satisfies Condition (FS) and that the ideal $J=A_1 A_{-1}$ is the maximal faithful, $\psi$-compatible ideal of $A_0$. Hence, the Cuntz-Pimsner representation is well-defined and equal to $(\iota_{A_{-1}}^J, \iota_{A_1}^J, \iota_{A_0}^J, \mathcal{O}_{(A_{-1},A_1, \psi')}(J))$. Moreover, they show that the graded representation $(i_{A_1}, i_{A_{-1}}, i_{A_0}, A)$ is Cuntz-Pimsner invariant with respect to $J$. By the universal property of relative Cuntz-Pimsner rings (see \cite[Thm. 3.18]{carlsen2011algebraic}), there exists a surjective map $\eta \colon (\iota_{A_{-1}}^{CP}, \iota_{A_1}^{CP}, \iota_{A_0}^{CP}, \mathcal{O}_{(A_{-1},A_1, \psi')}) \to (i_{A_1}, i_{A_{-1}}, i_{A_0}, A)$. It follows by Lemma \ref{lem:rep_maps}, that $\eta \colon \mathcal{O}_{(A_{-1}, A_1, \psi)} \to A$ is $\mathbb{Z}$-graded. By the graded uniqueness theorem for Cuntz-Pimsner rings (see \cite[Cor. 5.4]{carlsen2011algebraic}), it follows that $\eta$ is also injective. Thus, (\ref{eq:rep_iso}) holds. Note that $A \cong_{\text{gr}} \mathcal{O}_{(A_{-1},A_1, \psi')})$ follows from Corollary \ref{cor:rep_iso}.
\end{proof}

Let $R$ be a ring, let $(P,Q,\psi)$ be an $R$-system and let $(S,T,\sigma, B)$ be a graded covariant representation of $(P,Q,\psi)$. Recall (see Definition \ref{def:covariant_representation}) that for every $k \geq 0$ and $q \in Q^{\otimes k}, p \in P^{\otimes k}$ we have that $\sigma(\psi_k(p \otimes q)) = S^{\otimes k}(p) T^{\otimes k}(q)$. Since $S^{\otimes k}(p) \in B_{-k}$ and $T^{\otimes k}(q) \in B_k$, it follows that, $\sigma (\psi_k(p \otimes q)) \in B_{-k} B_k$. Moreover, since $I_{\psi,\sigma}^{(k)}$ is generated as a $B_0$-ideal by the set $\{ \sigma (\psi_k(p \otimes q)) \mid p \in P^{\otimes k}, q \in Q^{\otimes k} \}$, we have that $I_{\psi,\sigma}^{(k)} \subseteq B_{-k} B_k$.
 Recall (see Definition \ref{def:semi-full}) that we call $(S,T, \sigma, B)$ semi-full if $I_{\psi,\sigma}^{(k)} = B_{-k} B_k $ for every $k \geq 0$.   The following result is one of the key insights of this article:

\begin{proposition}
The covariant representation $$ (i_{A_{-1}}, i_{A_1}, i_{A_0}, A) \cong_{\text{r}} (\iota_{A_{-1}}^J, \iota_{A_1}^J, \iota_{A_0}^J, \mathcal{O}_{(A_{-1},A_1, \psi')}(J)) = (\iota_{A_{-1}}^{CP}, \iota_{A_1}^{CP}, \iota_{A_0}^{CP}, \mathcal{O}_{(A_{-1},A_1, \psi')})$$ in Theorem \ref{thm:clark} is a semi-full covariant representation of $(A_{-1}, A_1, \psi)$.
\label{rem:5}
\end{proposition}
\begin{proof} 
Note that $A$ comes equipped with a $\mathbb{Z}$-grading which trivially satisfies $i_{A_{-1}}(A_{-1}) \subseteq A_{-1}$, $i_{A_{1}}(A_{1}) \subseteq A_{1}$ and $i_{A_{0}}(A_{0}) \subseteq A_{0}$. Hence, $(i_{A_{-1}}, i_{A_1}, i_{A_0}, A)$ is a graded representation of $(A_{-1}, A_1,\psi)$. Note that $I_{\psi,i_{A_0}}^{(k)}  \subseteq A_{-k} A_k$.  Recall that $A$ is semi-saturated by Proposition \ref{prop:toeplitz_exponent}(b). Thus, for any monomial $a' a \in A_{-k} A_k$, we have that $a' = a_1' a_2' \dots a_k'$ and $a=a_1 a_2 \dots a_k$ for some elements $a_i' \in A_{-1}$ and $a_i \in A_1$. Next, note that by the definition, $$\psi_k((a_1' \otimes a_2' \otimes \dots \otimes a_k') \otimes (a_1 \otimes \dots a_k)) = a_1' a_2' \dots a_k' a_1 \dots a_k = a' a.$$ Thus, $A_{-k} A_k = I_{\psi,i_{A_0}}^{(k)}.$ For $k=0$, note that $\text{Im}(\psi_0)=A_0^2$ since $\psi_0(r \otimes r')=r r'$ for all $r,r' \in A_0$ by convention. Thus, we have that $A_0 A_0 = A_0^2 = i_{A_0}(A_0^2) = I_{\psi, i_{A_0}}^{(0)}$.  Hence, it follows that $I_{\psi,i_{A_0}}^{(k)} = A_{-k} A_k$ for every integer $k \geq 0$.
\end{proof}
\begin{remark}
In particular, Proposition \ref{rem:5} implies that some of the examples Clark, Fletcher, Hazrat and Li gave in \cite{2018arXiv180810114O} are realizable from semi-full representations. More precisely, the corner skew Laurent polynomial rings (see \cite[Expl. 3.4]{2018arXiv180810114O}) and the Steinberg algebras associated to unperforated graded groupoids (see \cite[Cor. 4.6]{2018arXiv180810114O}) are realizable as the representation ring belonging to a semi-full covariant representation.
\end{remark}

We will see that, for our purposes, we only need to consider s-unital and unital $R$-systems. In the $C^*$-setting, Chirvasitu \cite{2018arXiv180512318C} only considered unital $C^*$-correspondences (i.e. the coefficient $C^*$-algebra $A$ is unital). This assumption guarantees that the Cuntz-Pimsner $C^*$-algebra is unital. We analogously introduce the following notions for $R$-systems:
\begin{definition}
Let $R$ be a ring and let $(P,Q,\psi)$ be an $R$-system. The $R$-system $(P,Q,\psi)$ is called \emph{s-unital} if $R$ is an s-unital ring and $P,Q$ are s-unital $R$-bimodules. The $R$-system $(P,Q,\psi)$ is called \emph{unital} if $R$ is a unital ring and $P,Q$ are unital $R$-bimodules.
\label{def:s-unital}
\end{definition}

\begin{remark}
At this point we make two remarks.
\begin{enumerate}[(a)]
\begin{item}
Note that we explicitly require that $R$ is an s-unital (unital) ring for the $R$-system $(P,Q,\psi)$ to be s-unital (unital). This is needed since the trivial module $\{ 0 \}$ is a unital $R$-bimodule for any ring $R$ (cf. Example \ref{ex:2}).
\end{item}
\begin{item}
Let $R$ be a unital ring, let $(P,Q,\psi)$ be a unital $R$-system and let $(S,T,\sigma,B)$ be a covariant representation of $(P,Q,\psi)$. If $1_R$ is the multiplicative identity element of $R$, then $1_B = \sigma(1_R)$ is the multiplicative identity element of $B$.
\end{item}
\end{enumerate}
\end{remark}


We now show that a certain type of semi-saturated, nearly epsilon-strongly $\mathbb{Z}$-graded rings can be realized as Cuntz-Pimsner rings coming from s-unital $R$-systems.

\begin{definition}
If $A = \bigoplus_{i \in \mathbb{Z}} A_i$ is a semi-saturated, nearly epsilon-strongly $\mathbb{Z}$-graded ring that satisfies $\Ann_{A_0}(A_1) \cap (\Ann_{A_0}(A_1))^\bot = \{ 0 \}$, then $A$ is called \emph{pre-CP}.
\label{def:pre-cp}
\end{definition}


%

As a special case of Theorem \ref{thm:clark}, we obtain the following:
\begin{corollary}
Let $A = \bigoplus_{i \in \mathbb{Z}} A_i$ be a pre-CP ring. Let $\psi \colon A_{-1} \otimes A_1 \to A_0$ be defined by $a \otimes b \mapsto ab$. Then $(A_{-1}, A_1, \psi)$ is an s-unital $A_0$-system that satisfies Condition (FS) and 
\begin{equation}
(i_{A_{-1}}, i_{A_1}, i_{A_0}, A) \cong_{\text{r}} (\iota_{A_{-1}}^{CP}, \iota_{A_1}^{CP}, \iota_{A_0}^{CP}, \mathcal{O}_{(A_{-1}, A_1, \psi)}).
\label{eq:5}
\end{equation}
In particular, $A \cong_{\text{gr}} \mathcal{O}_{(A_{-1},A_1,\psi)}$. Furthermore, the covariant representation (\ref{eq:5}) is semi-full.  
\label{cor:nearly_cuntz}
\end{corollary}
\begin{proof}
Note that conditions (a) and (c) in Theorem \ref{thm:clark} are satisfied by definition. Moreover, by the assumption that $A$ is nearly epsilon-strongly $\mathbb{Z}$-graded (see Definition \ref{def:nystedt_epsilon}), it follows that $A_1$ is an s-unital $A_1 A_{-1} \text{--} A_{-1}A_{1}$-bimodule. From this, (b) follows directly. Furthermore, we see that $(A_{-1}, A_1, \psi)$ is an s-unital $A_0$-system.  The conclusion now follows by applying Theorem \ref{thm:clark} and Proposition \ref{rem:5}.
\end{proof}

Next, we give two sets of sufficient conditions for a ring to be pre-CP. Recall that a ring is called \emph{semi-prime} if it has no nonzero nilpotent ideals. 

\begin{lemma}
Let $A = \bigoplus_{i \in \mathbb{Z}} A_i$ be a $\mathbb{Z}$-graded ring. The following assertions hold:

\begin{enumerate}[(a)]
\begin{item}
  If $A_0$ is semi-prime, then $\Ann_{A_0}(A_1) \cap (\Ann_{A_0}(A_1))^\bot = \{ 0 \}$. If $A$ is semi-saturated, nearly epsilon-strongly $\mathbb{Z}$-graded and $A_0$ is semi-prime, then $A$ is pre-CP.
\end{item}
\begin{item}
If $A$ is unital strongly $\mathbb{Z}$-graded, then $A$ is pre-CP. 
\end{item}
\end{enumerate}
\label{lem:semi-prime}
\end{lemma}
\begin{proof}
%
(a): Note that $\Ann_{A_0}(A_1) \cap (\Ann_{A_0}(A_1))^\bot$ is a nilpotent ideal of $A_0$.


(b): Since $A$ is unital strongly $\mathbb{Z}$-graded, it follows that $A_i = (A_1)^i, A_{-i} = (A_{-1})^i$ for $i > 0$. Hence, $A$ is semi-saturated. Moreover, $\Ann_{A_0}(A_1) \subseteq \Ann_{A_0}(A_1A_{-1}) = \Ann_{A_0}(A_0) = \{ 0 \}$ since $A_0$ is unital. It follows that $\Ann_{A_0}(A_1) \cap (\Ann_{A_0}(A_1))^\bot = \{ 0 \}$. Finally, recall that unital strongly $\mathbb{Z}$-graded rings are nearly epsilon-strongly $\mathbb{Z}$-graded (see (\ref{eq:implications})). Thus, $A$ is pre-CP.
\end{proof}

\begin{proposition}
Let $K$ be a unital ring and let $E$ be any directed graph. Then the Leavitt path algebra $L_K(E)$ is pre-CP.
\label{prop:lpa-cp}
\end{proposition}
\begin{proof}
The Leavitt path algebra $L_K(E)$ is nearly epsilon-strongly $\mathbb{Z}$-graded (see \cite[Thm. 1.3]{nystedt2017epsilon}). Moreover, since  $L_K(E)$ can be realized as a Cuntz-Pimsner ring (see Section \ref{sec:lpa}), it follows by Proposition \ref{prop:toeplitz_exponent}(b) that $L_K(E)$ is semi-saturated. 
Next, we prove that,
\begin{equation}
\Ann_{L_K(E)_0}(L_K(E)_1) = \Span_K \{ v \in E^0 \mid v E^1 = \{ 0 \} \}.
\label{eq:71}
\end{equation} 
Since $L_K(E)_1 L_K(E)_{-1}$ is s-unital by Proposition \ref{prop:nearly_char}(a) and, $$\Ann_{L_K(E)_0}(L_K(E)_1)  \subseteq \Ann_{L_K(E)_0}(L_K(E)_1 L_K(E)_{-1}),$$ it follows that, 
\begin{equation}
L_K(E)_1 L_K(E)_{-1} \cap \Ann_{L_K(E)_0}(L_K(E)_1) \subseteq \Ann_{L_K(E)_1 L_K(E)_{-1}}(L_K(E)_1 L_K(E)_{-1}) = \{ 0 \}.
\label{eq:81}
\end{equation}
Furthermore, recall that the natural $\mathbb{Z}$-grading of $L_K(E)$ is given by, 
\begin{equation*}
L_K(E)_i = \Span_K \{ \alpha \beta^* \mid \alpha, \beta \in \text{Path}(E), \text{len}(\alpha) - \text{len}(\beta) =i \},
\end{equation*}
for all $i \in \mathbb{Z}$. By convention, the elements $v \in L_K(E)_0$ are considered to be paths of zero length. This means that $L_K(E)_0$ is generated by the sets $E^0$ and $B:=\{ \alpha \beta^* \mid \text{len}(\alpha)=\text{len}(\beta) \geq 1 \}.$ Any $\alpha \beta^* \in B$ can be written $\alpha \beta^* = f_1 \alpha' (\beta')^* (f_2)^* \in L_K(E)_1 L_K(E)_0 L_K(E)_{-1} = L_K(E)_1 L_K(E)_{-1} $ for some $f_1, f_2 \in E^1$ and $\alpha, \beta \in \text{Path}(E)$. Thus, $B \subseteq L_K(E)_1 L_K(E)_{-1}$.  By (\ref{eq:81}), it follows that $\Ann_{L_K(E)_0}(L_K(E)_1) \subseteq \Span_K \{ v \in E^0 \}$. 

To establish (\ref{eq:71}), it remains to prove that for any $v \in E^0$, we have that $v L_K(E)_1 = \{ 0 \}$ if and only if $v E^1 = \{ 0 \}$. The `only if' direction is clear since $E^1 \subseteq L_K(E)_1$. On the other hand, let $v \in E^0$ such that $v E^1 = \{0 \}$. Note that any $\alpha \beta^* \in L_K(E)_1$ satisfies $\text{len}(\alpha) - \text{len}(\beta) = 1$ which implies that $\text{len}(\alpha) \geq 1$. Hence, we can write $\alpha = f' \alpha'$ for some $f' \in E^1$ and some $\alpha' \in \text{Path}(E)$. It follows that $v \alpha \beta^*  = (v f') \alpha' \beta^* =0 $. Hence, $v L_K(E)_1 = \{ 0 \}$.

A moment's thought yields that, $$(\Ann_{L_K(E)_0}(L_K(E)_1))^\bot \cap \Span_K \{ v \in E^0 \} = \Span_K \{ v \in E^0 \mid vE^1 \ne \{ 0 \} \} .$$ Hence, $\Ann_{L_K(E)_0}(L_K(E)_1) \cap (\Ann_{L_K(E)_0}(L_K(E)_1))^\bot = \{ 0 \}$ and $L_K(E)$ is pre-CP.
\end{proof}

From Corollary \ref{cor:nearly_cuntz}, we derive necessary conditions for certain Cuntz-Pimsner rings to be nearly epsilon-strongly $\mathbb{Z}$-graded.

\begin{corollary}
Let $(P,Q,\psi)$ be an $R$-system such that (i) $\mathcal{O}_{(P,Q,\psi)}=\bigoplus_{i \in \mathbb{Z}} \mathcal{O}_i$ exists and is nearly epsilon-strongly $\mathbb{Z}$-graded and (ii) $ \Ann_{\mathcal{O}_0}(\mathcal{O}_1) \cap (\Ann_{\mathcal{O}_0}(\mathcal{O}_1))^\bot = \{ 0 \}$.  

\noindent

Let $\psi' \colon \mathcal{O}_{-1} \otimes \mathcal{O}_1 \to \mathcal{O}_0$ be defined by $\psi'(a \otimes a') = a a'$. Then $(\mathcal{O}_{-1}, \mathcal{O}_{1}, \psi')$ is an s-unital $\mathcal{O}_0$-system such that,
$$ (i_{\mathcal{O}_{-1}}, i_{\mathcal{O}_1}, i_{\mathcal{O}_0}, \mathcal{O}_{(P,Q,\psi)}) \cong_{\text{r}} (\iota_{\mathcal{O}_{-1}}^{CP}, \iota_{\mathcal{O}_{1}}^{CP}, \iota_{\mathcal{O}_{0}}^{CP}, \mathcal{O}_{(\mathcal{O}_{-1}, \mathcal{O}_1, \psi')}).$$ In particular, $\mathcal{O}_{(P,Q,\psi)} \cong_{\text{gr}} \mathcal{O}_{(\mathcal{O}_{-1}, \mathcal{O}_1, \psi')}$  Furthermore, the following assertions hold:
\begin{enumerate}[(a)]
\begin{item}
$(\mathcal{O}_{-1}, \mathcal{O}_{1}, \psi')$ is an s-unital $\mathcal{O}_0$-system that satisfies Condition (FS);
\end{item}
\begin{item}
$(\iota_{\mathcal{O}_{-1}}^{CP}, \iota_{\mathcal{O}_{1}}^{CP}, \iota_{\mathcal{O}_{0}}^{CP}, \mathcal{O}_{(\mathcal{O}_{-1}, \mathcal{O}_1, \psi')})$ is a semi-full covariant representation of $(\mathcal{O}_{-1}, \mathcal{O}_{1}, \psi')$;
\end{item}
\begin{item}
$I_{\psi',\iota_{\mathcal{O}}^{CP}}^{(k)} = \mathcal{O}_{-k} \mathcal{O}_k$ is s-unital for $k \geq 0$. 
\end{item}
\end{enumerate} 
\label{cor:reduction}
\end{corollary}
\begin{proof}
By Proposition \ref{prop:toeplitz_exponent}, $\mathcal{O}_{(P,Q,\psi)}$ is semi-saturated. Hence, with (i) and (ii), it follows that $\mathcal{O}_{(P,Q,\psi)}$ is pre-CP. Thus, Corollary \ref{cor:nearly_cuntz} establishes the isomorphism of covariant representations and the conclusions (a), (b). Since the covariant representation is semi-full we have that $I_{\psi',\iota_{\mathcal{O}}^{CP}}^{(k)} = \mathcal{O}_{-k} \mathcal{O}_k$ for each $k \geq 0$. By (i) and Proposition \ref{prop:nearly_char}(a), we see that $\mathcal{O}_{-k} \mathcal{O}_k$ is s-unital for every $k \geq 0$. Thus, (c) is established. 
\end{proof}

\begin{remark}
It is not clear to the author if the assumption (ii) in Corollary \ref{cor:reduction} is needed. No examples of nearly epsilon-strongly $\mathbb{Z}$-graded Cuntz-Pimsner rings that do not satisfy  $\Ann_{\mathcal{O}_0}(\mathcal{O}_1) \cap (\Ann_{\mathcal{O}_0}(\mathcal{O}_1))^\bot = \{ 0 \}$ have been found. On the other hand, it follows from Lemma \ref{lem:semi-prime} that condition (ii) in Corollary \ref{cor:reduction} is satisfied if either $\mathcal{O}_0$ is semi-prime or $\mathcal{O}_{(P,Q,\psi)}$ is strongly $\mathbb{Z}$-graded. 
\end{remark}

\section{Strongly $\mathbb{Z}$-graded Cuntz-Pimsner rings}
\label{sec:strongly}

In this section, we will provide sufficient conditions for the Toeplitz and Cuntz-Pimsner rings to be strongly $\mathbb{Z}$-graded. This is an algebraic analogue of recent work by Chirvasitu  \cite{2018arXiv180512318C} where he gave necessary and sufficient conditions for the gauge action of a Cuntz-Pimsner $C^*$-algebra to be free. Unfortunately, his proofs rely on topological arguments which do not seem to generalize fully to the algebraic setting.  

We begin by introducing the following new condition that is stronger than Condition (FS):

\begin{definition}
Let $R$ be a ring. An $R$-system $(P,Q,\psi)$ is said to satisfy \emph{Condition (FS')}  if there exist some $\Theta \in \mathcal{F}_P(Q)$ and $\Phi \in \mathcal{F}_Q(P)$ such that $\Theta(q) = q$ and $\Phi(p)=p$ for every $q \in Q$ and $p \in P$.
\label{def:cond_fsprime}
\end{definition}
We will later give an example (see Example \ref{ex:fsprime}) which shows that Condition (FS) and Condition (FS') are in fact different. 
We omit the proof of the following proposition as it is a straightforward analogue of the corresponding statement for Condition (FS). 

\begin{proposition}(cf. \cite[Lem. 3.8]{carlsen2011algebraic})
Let $R$ be a ring and let $(P,Q,\psi)$ be an $R$-system. If $(P,Q,\psi)$ satisfies condition (FS'), then $(P^{\otimes n}, Q^{\otimes n}, \psi_n)$ satisfies condition (FS') for every integer $n \geq 1$. 
\label{prop:fsprime}
\end{proposition}

Throughout the rest of this section, we assume that $R$ is a unital ring and that $(P,Q,\psi)$ is a unital $R$-system. The following result characterizes Condition (FS'):

\begin{proposition}
Let $R$ be a unital ring and let $(P,Q,\psi)$ be a unital $R$-system. The following assertions are equivalent:
\begin{enumerate}[(a)]
\begin{item}
$(P,Q,\psi)$ satisfies Condition (FS');
\end{item}
\begin{item}
$\id_Q = \Delta(1_R) \in \mathcal{F}_P(Q)$ and $\id_P = \Gamma(1_R) \in \mathcal{F}_Q(P)$. In this case, $\mathcal{L}_P(Q) = \mathcal{F}_P(Q)$ and $\mathcal{L}_Q(P) = \mathcal{F}_Q(P)$ are unital rings;
\end{item}
\begin{item}
$(P,Q,\psi)$ satisfies Condition (FS), $Q_R$ is finitely generated as a right $R$-module and $_R P$ is finitely generated as a left $R$-module.
\end{item}
\end{enumerate}
\label{prop:fsprime_char}
\end{proposition}
\begin{proof}
(a) $\Leftrightarrow$ (b): Consider the inclusions in (\ref{eq:inc1}). If $1_R$ is the multiplicative identity element of $R$, then $\id_Q = \Delta(1_R) \in \mathcal{L}_P(Q)$ is the multiplicative identity element for the ring $\mathcal{L}_P(Q)$. First assume that $(P,Q,\psi)$ satisfies Condition (FS'). Then, $\Theta \in \mathcal{F}_P(Q)$ is a multiplicative identity element of the ring $\mathcal{L}_P(Q)$. Hence, $\Theta = \Delta(1_R) = \id_Q$ which implies that $\mathcal{L}_P(Q) = \mathcal{F}_P(Q)$. Similarly, $\Phi = \Gamma(1_R) = \id_P$ which implies that $\mathcal{L}_Q(P)=\mathcal{F}_Q(P)$.  The converse statement follows by noting that $\Delta(1_R)(q)=1_R \cdot q = q$ and $\Gamma(1_R)(p) = p \cdot 1_R = p$ for all $q \in Q$ and $p \in P$.

(b) $\Rightarrow$ (c): Assume that $\text{id}_P(Q) \in \mathcal{F}_P(Q)$ and $\text{id}_Q(P) \in \mathcal{F}_Q(P)$. By choosing $\Theta := \text{id}_P(Q)$ and $\Phi := \text{id}_Q(P)$ in Definition \ref{def:cond_fs}, we see that $(P,Q,\psi)$ satisfies Condition (FS). Furthermore, there are some $q_1, \dots, q_n \in Q$ and $p_1, \dots, p_n \in P$ such that $ \text{id}_P(Q) = \sum_{i=1}^n \Theta_{q_i, p_i}$. For any $q' \in Q$ we then have that, $$q' =\text{id}_P(Q)(q') = \sum_{i=1}^n \Theta_{q_i,p_i}(q') = \sum_{i=1}^n q_i \cdot \psi(p_i \otimes q') \in \text{Span}_R \{ q_1, \dots, q_n \}. $$ In other words, $Q$ is finitely generated as a right $R$-module by the set $\{ q_1, \dots, q_n \}$. A similar argument establishes that $P$ is finitely generated as a left $R$-module. 

(c) $\Rightarrow$ (a): Assume that $(P,Q,\psi)$ satisfies Condition (FS), $Q$ is generated as a right $R$-module by the set $\{q_1, \dots, q_n \}$ and that $P$ is generated as a left $R$-module by the set $\{p_1, \dots, p_m \}$ for some non-negative integers $n,m$ and $q_i \in Q, p_i \in P$. Let $\Theta \in \mathcal{F}_P(Q)$ and $\Phi \in \mathcal{F}_Q(P)$ be such that $\Theta(q_i) = q_i$ and $\Phi(p_j) = p_j$ for all $i \in \{ 1, \dots, n\}, j \in \{1, \dots, m\}.$ Take an arbitrary $q' \in Q$ and note that there are some $r_i \in R$ such that $q' = \sum_{i=1}^n q_i \cdot r_i$. But since $\Theta$ is a right $R$-module homomorphism, it follows that $\Theta(q') = \Theta(\sum_{i=1}^n q_i \cdot r_i) = \sum_{i=1}^n \Theta(q_i) \cdot r_i = \sum_{i=1}^n q_i \cdot r_i = q'$. A similar argument shows that $\Phi(p')=p'$ for every $p \in P$. Thus, $(P,Q,\psi)$ satisfies Condition (FS'). 
\end{proof}
\begin{remark}
At this point, we make two remarks regarding Proposition \ref{prop:fsprime_char}.
\begin{enumerate}[(a)]
\begin{item}
Note that Condition (FS) (cf. Definition \ref{def:cond_fs}) and Condition (FS') (cf. Definition \ref{def:cond_fsprime}) relates to each other similarly to how s-unital rings relate to unital rings. In Section \ref{sec:epsilon}, we will show that  Condition (FS)/Condition (FS') implies that the ideals $\mathcal{T}_i \mathcal{T}_{-i}$ are s-unital/unital for $i \geq 0$.
\end{item}
\begin{item}
In the $C^*$-setting, finite generation of the Hilbert module $E$ is equivalent to the ring of compact operators $B(E)=K(E)$ being unital. Proposition \ref{prop:fsprime_char} is the algebraic analogue of this statement.
\end{item}
\end{enumerate}
\end{remark}

The following system satisfies Condition (FS) but not Condition (FS'):

\begin{example}

Let $E$ consist of one vertex $v$ with countably infinitely many loops $f_1, f_2, \dots$. 
This is sometimes called a rose with countably many petals. 
\begin{displaymath}
	\xymatrix{
	\bullet_v \ar@(ur,ul)_{(\infty)}
	}
\end{displaymath}
The standard Leavitt path algebra system $(P,Q,\psi)$ attached to the graph $E$ satisfies Condition (FS) (see \cite[Expl. 5.8]{carlsen2011algebraic}). Furthermore, it is straightforward to check that $(P,Q,\psi)$ is a unital $R$-system with multiplicative identity element $1_R = \eta_{v}$.  
However, since $E$ contains infinitely many edges it follows that $P$ and $Q$ are not finitely generated (see Section \ref{sec:lpa} and Lemma \ref{lem:finitely_many_edges}). By Proposition \ref{prop:fsprime_char}(c), $(P,Q,\psi)$ does not satisfy Condition (FS'). In other words, $(P,Q,\psi)$ is an example of an $R$-system satisfying Condition (FS) but not Condition (FS').


\label{ex:fsprime}
\end{example}

%

 To prove that the Toeplitz ring is strongly $\mathbb{Z}$-graded, we need the following definition.


\begin{definition}
Let $R$ be a unital ring, let $(P,Q,\psi)$ be an $R$-system satisfying Condition (FS') and let $(S,T,\sigma,B)$ be a covariant representation of $(P,Q,\psi)$. Then $(S,T,\sigma,B)$ is called \emph{faithful} if $\pi_{T,S}(\Delta(1_R)) = \sigma(1_R)$.
\label{def:faithful}
\end{definition}
 To make sense of Definition \ref{def:faithful}, note that $\Delta(1_R) \in \mathcal{F}_P(Q)$ for every $R$-system satisfying Condition (FS') by Proposition \ref{prop:fsprime_char}(b). Hence, the condition $\pi_{T,S}(\Delta(1_R)) = \sigma(1_R)$ makes sense. It also follows from Proposition \ref{prop:fsprime_char}(c) that if an $R$-system $(P,Q,\psi)$ admits a faithful covariant representation, then $Q$ is finitely generated as a right $R$-module and $P$ is finitely generated as a left $R$-module. 

\smallskip

Next, we will consider a graded covariant representation and derive sufficient conditions for it to be strongly $\mathbb{Z}$-graded.

\begin{lemma}
Let $R$ be a unital ring. Suppose that $(P,Q,\psi)$ is an $R$-system and that $(S,T, \sigma, B)$ is a graded, injective, surjective and faithful representation of $(P,Q,\psi)$. Then, $$\pi_{T^n, S^n}(\Delta^{(n)}(1_R))=\sigma(1_R) = 1_B \in B_n B_{-n}$$ for every $n > 0$. 
\label{lem:faithful}
\end{lemma}
\begin{proof}
Take an arbitrary $n > 0$. By Proposition \ref{prop:fsprime}, $(P^{\otimes n}, Q^{\otimes n}, \psi_n)$ satisfies Condition (FS'). This means that $\Delta^{(n)}(1_R) \in \mathcal{F}_{P^{\otimes n}}(Q^{\otimes n})$. Furthermore, by faithfulness, $\pi_{T,S}(\Delta(1_R))= \sigma(1_R) = \sum_i T(q_i) S(p_i)$ for some $q_i \in Q, p_i \in P$. Then, 
\begin{align*}
1_B &= \sigma(1_R) = \sum_i T(q_i) (1_B) S(p_i) = \sum_{i,j} T(q_i) T(q_j) S(p_j) S(p_i)  \\ &\in \Span_R \{ T^2(q) S^2(p) \mid q \in Q^{\otimes 2}, p \in P^{\otimes 2} \} \subseteq B_2 B_{-2}.
\end{align*}
By an induction argument, we get that, $$1_B \in \Span_R \{ T^n(q) S^n(p) \mid q \in Q^{\otimes n}, p \in P^{\otimes n} \} \subseteq B_n B_{-n}$$ for any $n > 0$. By Proposition \ref{prop:pi} and the assumption that the covariant representation is injective, it follows that the map $\pi_{T^n, S^n} \colon \mathcal{F}_{P^{\otimes n}}(Q^{\otimes n}) \to \Span_R \{ T(q) S(p) \mid q \in Q^{\otimes n}, p \in P^{\otimes n} \}$ is a ring isomorphism. Hence, $\pi_{T^n, S^n}(\Delta^{(n)}(1_R)) = 1_B = \sigma(1_R) \in B_{n} B_{-n}$ for $n > 0$. 
\end{proof}

\begin{lemma}
Let $R$ be a unital ring and let $(P,Q,\psi)$ be a unital $R$-system such that the map $\psi \colon P \otimes Q \to R$ is surjective. Let $(S,T,\sigma, B)$ be a surjective, graded covariant representation of $(P,Q,\psi)$. Then, $1_B \in B_{-n} B_n$ for every $n > 0$.
\label{lem:psi_surjective}
\end{lemma}
\begin{proof}
We prove that if $\psi \colon P \otimes Q \to R$ is surjective, then $\psi_n \colon P^{\otimes n} \otimes Q^{\otimes n} \to R$ is surjective for every $n > 1$. The proof goes by induction on $n$. Suppose that $\psi_{n-1}$ is surjective. Then there is some $p \in P^{\otimes (n-1)}$ and $q \in Q^{\otimes (n-1)}$ such that $\psi_{n-1}(p \otimes q) = 1_R$. Then, since $1_R$ acts trivially on $Q$, it follows that, $$\psi_n((p' \otimes p) \otimes (q \otimes q')) = \psi(p' \otimes \psi_{n-1}(p \otimes q) \cdot q') = \psi(p' \otimes 1_R \cdot q') = \psi(p' \otimes q') = 1_R,$$ if we choose $p'$ and $q'$ such that $\psi(p' \otimes q') = 1_R$. Thus, the claim follows from the induction principle. 

Take an arbitrary integer $n > 0$. We have that $1_R = \psi_n(p \otimes q)$ for some $p \in P^{\otimes n}$ and $q \in Q^{\otimes n}$. Hence, $\sigma(1_R) = \sigma(\psi_n (p \otimes q)) = S^n(p) T^n(q)  \in B_{-n} B_n$ which proves that $1_B = \sigma(1_R) \in B_{-n} B_n$ for every $n > 0$. 
\end{proof}

We have now found sufficient conditions for a representation ring to be strongly $\mathbb{Z}$-graded:

\begin{proposition}
Let  $R$ be a unital ring and let $(P,Q,\psi)$ be a unital $R$-system that satisfies Condition (FS'). Let $(S,T,\sigma,B)$ be an injective, surjective and graded covariant representation of $(P,Q,\psi)$. Furthermore, suppose that the following assertions hold:
\begin{enumerate}[(a)]
\begin{item}
$(S, T, \sigma, B)$ is a faithful representation of $(P,Q,\psi)$;
\end{item}
\begin{item}
$\psi$ is surjective.
\end{item}
\end{enumerate}
Then $B$ is strongly $\mathbb{Z}$-graded.
\label{prop:strong_suff}
\end{proposition}
\begin{proof}
By assumption (a), it follows from Lemma \ref{lem:faithful} that $1_B \in B_n B_{-n}$ for every $n > 0$. By assumption (b) and Lemma \ref{lem:psi_surjective}, it follows that $1_B \in B_{-n} B_n$ for every $n > 0$. Furthermore, since $1_B = \sigma(1_R) \in B_0$, it follows that $B_0$ is a unital subring of $B$. Thus, $1_B \in \mathcal{T}_i \mathcal{T}_{-i}$ for every $i \in \mathbb{Z}$. It then follows that $B$ is strongly $\mathbb{Z}$-graded (see e.g. \cite[Prop. 1.1.1]{nastasescu2004methods}). 
\end{proof}

Note that since the Toeplitz representation $(\iota_P, \iota_Q, \iota_R, \mathcal{T}_{(P,Q,\psi)})$ is injective, surjective and graded, Proposition \ref{prop:strong_suff} gives, in particular, sufficient conditions for the Toeplitz ring to be strongly $\mathbb{Z}$-graded. 

\begin{corollary}
Let  $R$ be a unital ring and let $(P,Q,\psi)$ be a unital $R$-system that satisfies Condition (FS'). Consider the Toeplitz ring $\mathcal{T}_{(P,Q,\psi)} = \bigoplus_{i \in \mathbb{Z}} \mathcal{T}_i$. If $\pi(\Delta(1_R))=\iota_R(1_R)$ and $\psi$ is surjective, then $\mathcal{T}_{(P,Q,\psi)}$ is strongly $\mathbb{Z}$-graded.
\end{corollary}


The requirement of faithfulness is more easily formulated when considering the relative Cuntz-Pimsner representations.

\begin{corollary}
Let $R$ be a unital ring and let $(P,Q,\psi)$ be a unital $R$-system that satisfies Condition (FS'). Let $J \subseteq R$ be a $\psi$-compatible ideal. Furthermore, suppose that the following assertions hold:
\begin{enumerate}[(a)]
\begin{item}
 $1_R \in J$;
\end{item}
\begin{item} 
$\psi$ is surjective.
\end{item}
\end{enumerate} 
Then the relative Cuntz-Pimsner ring $\mathcal{O}_{(P,Q,\psi)}(J)$ is strongly $\mathbb{Z}$-graded.
\label{cor:cuntz_strongly}
\end{corollary}
\begin{proof}
Recall that the Cuntz-Pimsner representation $(\iota_P^J, \iota_Q^J, \iota_R^J, \mathcal{O}_{(P,Q,\psi)}(J))$ is injective, surjective and graded. Furthermore, note that (a) implies that the identity $\iota_R^J(1_R) = \pi_{\iota_Q^J, \iota_P^J}(\Delta(1_R))$ holds in the Cuntz-Pimsner ring. This implies that the representation $(\iota_P^{J}, \iota_Q^{J}, \iota_R^{J}, \mathcal{O}_{(P,Q,\psi)}(J))$ is faithful. By Proposition \ref{prop:strong_suff} and (b), we have that $\mathcal{O}_{(P,Q,\psi)}(J)$ is strongly $\mathbb{Z}$-graded.
\end{proof}

For the rest of this section, we apply the above theorems to the special cases of Leavitt path algebras and corner skew Laurent polynomial rings. We begin by proving that the conditions in Corollary \ref{cor:cuntz_strongly} are satisfied for any Leavitt path algebra associated to a finite graph without sinks. 

\begin{remark}
The Leavitt path algebra of a graph $E$ is the Cuntz-Pimsner ring relative to the ideal  $J$ generated by the regular vertices $\text{Reg}(E)\subseteq E^0$. In other words, $L_K(E) \cong_{\text{gr}} \mathcal{O}_{(P,Q,\psi)}(J)$ where $(P,Q,\psi)$ is the standard Leavitt path algebra system associated to $E$ (see \cite[Expl. 5.8]{carlsen2011algebraic} and Section \ref{sec:lpa}). Suppose that $E$ is a finite graph without any sinks. We now prove that the conditions (a) and (b) in Corollary \ref{cor:cuntz_strongly} are satisfied.
\begin{enumerate}[(a)]
\begin{item}
Since a singular vertex (non-regular vertex) is either an infinite emitter or a sink, by the requirements on $E$, it follows that $\text{Reg}(E) = E^0$. This implies that $J=R$ and hence that $1_R = \sum_{v \in E^0} \eta_v \in J$. 
\end{item}
\begin{item}
Since $E$ does not contain any sinks, we have that for any $v \in E^0$ there is some $f \in E^1$ such that $r(f) = v$. Thus, $\eta_v = \eta_{r(f)} = \psi(\eta_{f^*} \otimes \eta_f)$. This proves that $\psi$ is surjective. 
\end{item}
\end{enumerate}
\label{rem:aa1}
\end{remark}

Compare the following lemma with Example \ref{ex:fsprime}:

\begin{lemma}
Let $K$ be a unital ring and let $E$ be a directed graph with finitely many vertices. Then the standard Leavitt path algebra system $(P,Q,\psi)$ is a unital $R$-system. Furthermore, $(P,Q,\psi)$ satisfies Condition (FS') if and only if $E$ has finitely many edges. 
\label{lem:finitely_many_edges}
\end{lemma}
\begin{proof}
Recall that the standard Leavitt path algebra system (see Section \ref{sec:lpa}) is defined by $P = \bigoplus_{f \in E^1} K \eta_{f^*}$ and $Q=\bigoplus_{f \in E^1} K\eta_f$. The assumption that $E$ has finitely many vertices implies that $R$ is a unital ring and that $(P,Q,\psi)$ is a unital $R$-system. By Proposition \ref{prop:fsprime_char}(c), $(P,Q,\psi)$ satisfies Condition (FS') if and only if $(P,Q,\psi)$ satisfies Condition (FS), (i) $Q$ is finitely generated as a right $R$-module and (ii) $P$ is finitely generated as a left $R$-module. However, the $R$-system $(P,Q,\psi)$ always satisfies Condition (FS) (see \cite[Expl. 5.8]{carlsen2011algebraic}). Moreover, it follows from the definition of $P$ and $Q$ that (i) and (ii) hold if and only if $E$ has finitely many edges. 
\end{proof}


We can now partially recover a result obtained by Hazrat on when a Leavitt path algebra of a finite graph is strongly $\mathbb{Z}$-graded (see \cite[Thm. 3.15]{hazrat2013graded}).

\begin{corollary}
Let $K$ be a unital ring and let $E$ be a finite graph without any sinks. Then the Leavitt path algebra $L_K(E)$ is strongly $\mathbb{Z}$-graded.
\label{cor:lpa_strong}
\end{corollary}
\begin{proof}
By Lemma \ref{lem:finitely_many_edges}, Remark \ref{rem:aa1} and Corollary \ref{cor:cuntz_strongly} it follows that $L_K(E) \cong_{\text{gr}} \mathcal{O}_{(P,Q,\psi)}(J)$ is strongly $\mathbb{Z}$-graded.
\end{proof}

We will now consider corner skew Laurent polynomial rings. Recall that we need to specify a unital ring $R$, an idempotent $e \in R$ and a corner isomorphism $\alpha \colon R \to eRe$. Moreover, recall that an idempotent $e \in R$ is called \emph{full} if $ReR = R$. Hazrat showed (see \cite[Prop. 1.6.6]{hazrat2016graded}) that $R[t_{+}, t_{-}; \alpha]$ is strongly $\mathbb{Z}$-graded if and only if $e$ is a full idempotent. 

\begin{corollary}
Let $R$ be a unital ring and let $\alpha \colon R \to eRe$ be a ring isomorphism where $e$ is an idempotent of $R$. The corner skew Laurent polynomial ring $R[t_{+}, t_{-}; \alpha]$ is strongly $\mathbb{Z}$-graded if $e$ is a full idempotent.
\label{cor:fractional_strong}
\end{corollary}
\begin{proof}
Let $(P,Q,\psi)$ denote the $R$-system in \cite[Expl. 5.6]{carlsen2011algebraic}, i.e. let, $$P = \Big \{ \sum r_i \alpha(r_i') \mid r_i, r_i' \in R \Big \}, \quad Q = \Big \{ \sum \alpha(r_i) r_i' \mid r_i, r_i' \in R \Big \}, \quad \psi(p \otimes q) = pq,$$ where the left and right actions of $R$ on $P$ and $Q$ are defined by $r \cdot r_1 \alpha(r_2) = r r_1 \alpha(r_2)$, $r_1 \alpha(r_2) \cdot r = r_1 \alpha(r_2 r)$, $r \cdot \alpha(r_1) r_2 = \alpha(r r_1) r_2$, $\alpha(r_1) r_2 \cdot r = \alpha(r_1) r_2 r$ for all $r, r_1, r_2 \in R$. By \cite[Expl. 5.7]{carlsen2011algebraic}, the $R$-system $(P,Q,\psi)$ satisfies Condition (FS). Assume that $e$ is a full idempotent. Then, $$\text{Im}(\psi) = P Q = (R eR e) (e R e R) = R eR (e e ) R e R = (R e R) e (R e R) = R e R = R.$$ Hence, $\psi$ is surjective.  Furthermore, note that $_R P = {_R} ( R e R e) =  {_R} (R e R)e = {_R} R e$ as left $R$-modules. It follows that $P$ is finitely generated as a left $R$-module. Similarly, $Q _R = ( e R e R)_R = e R _R$ is finitely generated as a right $R$-module. By Proposition \ref{prop:fsprime_char}(c), it follows that $(P,Q,\psi)$ satisfies Condition (FS').  Recall from \cite[Expl. 5.7]{carlsen2011algebraic} that $J= R$ is $\psi$-compatible and $R[t_{+}, t_{-}; \alpha] \cong_{\text{gr}} \mathcal{O}_{(P,Q,\psi)}(J)$. By Corollary \ref{cor:cuntz_strongly}, it follows that $\mathcal{O}_{(P,Q,\psi)}(J)$ is strongly $\mathbb{Z}$-graded. Thus, $R[t_{+}, t_{-}; \alpha]$ is strongly $\mathbb{Z}$-graded.
\end{proof}

\section{Epsilon-strongly $\mathbb{Z}$-graded Cuntz-Pimsner rings}
\label{sec:epsilon}

We will show that Condition (FS) and Condition (FS') correspond to local unit properties of the rings $\mathcal{T}_i \mathcal{T}_{-i}$ for $i > 0$. This allows us to find sufficient conditions for certain representation rings to be nearly epsilon-strongly and epsilon-strongly $\mathbb{Z}$-graded.

\begin{proposition}
Let $R$ be an s-unital ring and let $(P,Q,\psi)$ be an s-unital $R$-system that satisfies Condition (FS). Consider the Toeplitz ring $\mathcal{T}_{(P,Q,\psi)}=\bigoplus_{i \in \mathbb{Z}} \mathcal{T}_i$. The following assertions hold:
\begin{enumerate}[(a)]
\begin{item}
For $i \geq 0$, $\mathcal{T}_i$ is a left s-unital $\mathcal{T}_i \mathcal{T}_{-i}$-module;
\end{item}
\begin{item}
For $i \geq 0$, $\mathcal{T}_{-i}$ is a right s-unital $\mathcal{T}_i \mathcal{T}_{-i}$-module;
\end{item}
\begin{item}
$\mathcal{T}_i \mathcal{T}_{-i}$ is an s-unital ring for $i \geq 0$;
\end{item}
\begin{item}
$\mathcal{T}_i = \mathcal{T}_i \mathcal{T}_{-i} \mathcal{T}_i $ for every $i \in \mathbb{Z}$. 
\end{item}
\end{enumerate}
\label{prop:nearly_toeplitz}
\end{proposition}
\begin{proof}
(a): Take an arbitrary integer $i \geq 0$ and an element $s \in \mathcal{T}_i$. Then, $ s = \sum_j \iota_{Q^{\otimes m_j}}(q_j) \iota_{P^{\otimes n_j}}(p_j)$ for some non-negative integers $\{ m_j \}, \{n_j\}$ and elements $q_j \in Q^{\otimes m_j}, p_j \in P^{\otimes n_j}$. Note that $m_j - n_j = i$ for all indices $j$. Furthermore, since $i$ is non-negative, we have that $0 \leq i \leq m_j$ for all $j$. We will construct an element $\epsilon(s)$ such that $\epsilon(s) s = s$. 

If $i=0$, then by the assumption that $(P,Q,\psi)$ is an s-unital $R$-system and Remark \ref{rem:s-unital}, we can find some element $r \in R$ such that $r \cdot q_j = q_j$ for all $j$. Put $\epsilon(s) := \iota_R(r) \in \mathcal{T}_0$. Then, 
\begin{align*}
\epsilon(s) s &= \iota_R(r) \sum_j \iota_{Q^{\otimes m_j}}(q_j) \iota_{P^{\otimes n_j}}(p_j) = \sum_j \iota_{Q^{\otimes m_j}}(r \cdot q_j) \iota_{P^{\otimes n_j}}(p_j) \\ &= \sum_j \iota_{Q^{\otimes m_j}}(q_j) \iota_{P^{\otimes n_j}}(p_j) = s.
\end{align*}

If $i > 0$, then let $q_j'$ denote the $i$th initial segment of $q_j$ for every $j$. In other words, for every $j$ we have that $q_j = q_j' \otimes q_j''$ where $q_j' \in Q^{\otimes i}$ and $q_j'' \in Q^{\otimes (m_j-i)}$. Since $(P,Q,\psi)$ satisfies Condition (FS), it follows by \cite[Lem. 3.8]{carlsen2011algebraic} that $(P^{\otimes i}, Q^{\otimes i}, \psi_i)$ satisfies Condition (FS). Therefore, there is some $\Theta \in \mathcal{F}_{P^{\otimes i}}(Q^{\otimes i})$ such that $\Theta(q_j') = q_j'$ for all $j$. Invoking Proposition \ref{prop:pi}, we put $\epsilon(s) :=  \pi_{\iota_{Q^{\otimes i}}, \iota_{P^{\otimes i}}}(\Theta).$ By Proposition \ref{prop:pi} and (\ref{eq:grading}), we have that, $$\pi_{\iota_{Q^{\otimes i}}, \iota_{P^{\otimes i}}}(\Theta) \in \Span_R \{ \iota_{Q^{\otimes i}}(q) \iota_{P^{\otimes i}}(p) \mid q \in Q^{\otimes i}, p \in P^{\otimes i} \} \subseteq \mathcal{T}_i \mathcal{T}_{-i}.$$ Furthermore, by using the left relation of (\ref{eq:2.9a}),
 
\begin{align*}
\epsilon(s) s &= \pi(\Theta) \sum_j \iota_{Q^{\otimes m_j}}(q_j) \iota_{P^{\otimes n_j}}(p_j) = \pi(\Theta) \sum_j \iota_{Q^{\otimes i}}(q_j') \iota_{Q^{\otimes (m_j-i)}}(q_j'') \iota_{P^{\otimes n_j}}(p_j) \\
&= \sum_j ( \pi(\Theta) \iota_{Q^{\otimes i}}(q_j') ) \iota_{Q^{\otimes (m_j-i)}}(q_j'') \iota_{P^{\otimes n_j}}(p_j) = \sum_j (\iota_{Q^{\otimes i}}(\Theta(q_j')) ) \iota_{Q^{\otimes (m_j-i)}}(q_j'') \iota_{P^{\otimes n_j}}(p_j) \\
&= \sum_j (\iota_{Q^{\otimes i}}(q_j') ) \iota_{Q^{\otimes (m_j-i)}}(q_j'') \iota_{P^{\otimes n_j}}(p_j) = \sum_j \iota_{Q^{\otimes m_j}}(q_j) \iota_{P^{\otimes n_j}}(p_j) = s.
\end{align*}

(b): Analogous to (a)

(c): Let $i \geq 0$ be an arbitrary non-negative integer. Any element of $\mathcal{T}_i \mathcal{T}_{-i}$ is a finite sum $s = \sum_j a_j b_j$ where $a_j \in \mathcal{T}_i$ and $b_j \in \mathcal{T}_{-i}$. Since $\mathcal{T}_i$ is a left s-unital $\mathcal{T}_i \mathcal{T}_{-i}$-module by (a), Remark \ref{rem:s-unital} implies that we can find some element $t_1 \in \mathcal{T}_i \mathcal{T}_{-i}$ such that $t_1 a_j = a_j$ for all indices $j$. Similarly, (b) and Remark \ref{rem:s-unital} implies that there is some element $t_2 \in \mathcal{T}_i \mathcal{T}_{-i}$ such that $b_j t_2 = b_j$ for all indices $j$. Hence, $t_1 s = s$ and $s t_2 = s$. This implies that $\mathcal{T}_i \mathcal{T}_{-i}$ is a left s-unital $\mathcal{T}_i \mathcal{T}_{-i}$-module and a right s-unital $\mathcal{T}_i \mathcal{T}_{-i}$-module. Thus, $\mathcal{T}_i \mathcal{T}_{-i}$ is an s-unital ring.

(d):  Take an arbitrary integer $i \in \mathbb{Z}$. From the grading, it is clear that $\mathcal{T}_i \mathcal{T}_{-i} \mathcal{T}_i \subseteq \mathcal{T}_i$. It remains to show that $\mathcal{T}_i \subseteq \mathcal{T}_i \mathcal{T}_{-i} \mathcal{T}_i$. Let $s \in \mathcal{T}_i$ be an arbitrary element. First suppose that $i \geq 0$, then by (a) there is some $\epsilon(s) \in \mathcal{T}_i \mathcal{T}_{-i}$ such that $s = \epsilon(s) s \in \mathcal{T}_i \mathcal{T}_{-i} \mathcal{T}_i$. On the other hand, if $i < 0$, then by (b) there is some $\epsilon(s) \in \mathcal{T}_{-i} \mathcal{T}_{i}$ such that $s=s \epsilon(s) \in \mathcal{T}_i \mathcal{T}_{-i} \mathcal{T}_i$. Thus, $\mathcal{T}_i = \mathcal{T}_i \mathcal{T}_{-i} \mathcal{T}_i$ for every $i \in \mathbb{Z}$. 
\end{proof}

Recall that for idempotents $e, f$ we define the idempotent ordering by $e \leq f \iff ef=fe=e$.
\begin{remark}

Let $A$ be an epsilon-strongly $\mathbb{Z}$-graded ring. Let $\epsilon_i \in A_i A_{-i}$ denote the multiplicative identity element of $A_i A_{-i}$ for $i \in \mathbb{Z}$ (see Proposition \ref{prop:nearly_char}). 
If the gradation on $A$ is semi-saturated, then
$\epsilon_0 \geq \epsilon_1 \geq \epsilon_2 \geq \epsilon_3 \geq \ldots$
and
$\epsilon_0 \geq \epsilon_{-1} \geq \epsilon_{-2} \geq \epsilon_{-3} \geq \ldots$.
\end{remark}

For the next section, let $(P,Q,\psi)$ be a unital $R$-system. Suppose that $(P,Q,\psi)$ satisfies Condition (FS'). By Proposition \ref{prop:fsprime_char}(b), this implies that $\Delta(1_R) \in \mathcal{F}_P(Q)$ and $\Gamma(1_R) \in \mathcal{F}_Q(P)$.  Consider the Toeplitz representation $(\iota_Q, \iota_P, \iota_R, \mathcal{T}_{(P,Q,\psi)})$. We define, $$\epsilon_0 := \iota_R(1_R), \qquad \epsilon_i := \pi_{\iota_{Q^{\otimes i}}, \iota_{P^{\otimes i}}}(\Delta^i(1_R)) = \chi_{\iota_{Q^{\otimes i}}, \iota_{P^{\otimes i}}}(\Gamma^i(1_R)) \in \mathcal{T}_{i} \mathcal{T}_{-i},$$ for $i > 0$.

\begin{lemma}
The sequence $\{ \epsilon_i \}_{i \geq 0}$ consists of idempotents such that $\epsilon_0 \geq \epsilon_1 \geq \epsilon_2 \geq \epsilon_3 \geq \epsilon_4 \geq \dots$ holds in the idempotent ordering. 
\label{lem:idem_seq}
\end{lemma}
\begin{proof}
Fix an arbitrary integer $i \geq 0$. By Proposition \ref{prop:pi}, we have that $\epsilon_i = \pi (\Delta^i(1_R)) = \sum_j \iota_{Q^{\otimes i}}(q_j) \iota_{P^{\otimes i}}(p_j)$ for some $q_j \in Q^{\otimes i}$ and $p_j \in P^{\otimes i}$. Then, by the left relation in (\ref{eq:2.9a}), 
\begin{align*}
\epsilon_i^2 &= \sum_j \epsilon_i \iota_{Q^{\otimes i}}(q_j) \iota_{P^{\otimes i}}(p_j) = \sum_j (\pi (\Delta^i(1_R)) \iota_{Q^{\otimes i}}(q_j)) \iota_{P^{\otimes i}}(p_j) \\ &= \sum_j \iota_{Q^{\otimes i}}(\Delta^i(1_R)(q_j))\iota_{P^{\otimes i}}(p_j) = \sum_j \iota_{Q^{\otimes i}}(q_j) \iota_{P^{\otimes i}}(p_j) = \epsilon_i.
\end{align*} Hence, $\epsilon_i$ is an idempotent.

It is clear that $\iota_R(1_R) = \epsilon_0 \geq \epsilon_1$. Take an arbitrary integer $m > 0$. We will prove that $\epsilon_m \geq \epsilon_{m+1}$. This is equivalent to $\epsilon_{m+1} = \epsilon_{m+1} \epsilon_m = \epsilon_m \epsilon_{m+1}$. 
We first prove that $\epsilon_m \epsilon_{m+1} = \epsilon_m$. 
Let $\epsilon_{m+1} = \sum_{j} \iota_{Q^{\otimes m+1}}(q_j)\iota_{P^{\otimes m+1}}(p_j)$. Write $q_j = q_j' \otimes q_j''$ where $q_j' \in Q^{\otimes m}$ and $q_j'' \in Q$. Then, by the left relation in (\ref{eq:2.9a}),
\begin{align*}
\epsilon_m \epsilon_{m+1} &= \sum_j \epsilon_m \iota_{Q^{\otimes m+1}}(q_j)\iota_{P^{\otimes m+1}}(p_j) = \sum_j \epsilon_m \iota_{Q^{\otimes m}}(q_j') \iota_{Q}(q_j'')\iota_{P^{\otimes m+1}}(p_j) \\ &= \sum_j \iota_{Q^{\otimes m}}(\Delta^{m}(1_R)(q_j')) \iota_Q (q_j'') \iota_{P^{\otimes m+1}}(p_j) = \sum_j \iota_{Q^{\otimes m}}(q_j') \iota_Q (q_j'') \iota_{P^{\otimes m+1}}(p_j) \\ &= \sum_j \iota_{Q^{\otimes m+1}}(q_j) \iota_{P^{\otimes m+1}}(p_j) = \epsilon_m.
\end{align*}

Again, let $\epsilon_{m+1} = \sum_{j} \iota_{Q^{\otimes m+1}}(q_j)\iota_{P^{\otimes m+1}}(p_j)$. This time write $p_j = p_j' \otimes p_j''$ for some $p_j' \in P$ and $p_j'' \in P^{\otimes m}$. Then, by the right relation in (\ref{eq:2.9a}),
\begin{align*}
\epsilon_{m+1} \epsilon_m  &= \sum_j \iota_{Q^{\otimes m+1}}(q_j)\iota_{P^{\otimes m+1}}(p_j) \epsilon_m  = \sum_j  \iota_{Q^{\otimes m+1}}(q_j) \iota_P(p_j') \iota_{P^{\otimes m}}(p_j'') \epsilon_m \\ &= \sum_j  \iota_{Q^{\otimes m+1}}(q_j) \iota_P(p_j') \iota_{P^{\otimes m}}(p_j'') \chi(\Gamma^m(1_R)) = \sum_j  \iota_{Q^{\otimes m+1}}(q_j) \iota_P(p_j') \iota_{P^{\otimes m}}(\Gamma^m(1_R)(p_j'')) \\ &= \sum_j  \iota_{Q^{\otimes m+1}}(q_j) \iota_P(p_j') \iota_{P^{\otimes m}}(p_j'') = \sum_j  \iota_{Q^{\otimes m+1}}(q_j) \iota_{P^{\otimes m +1 }}(p_j) = \epsilon_m.
\end{align*}

\end{proof}

\begin{proposition}
Let $R$ be a unital ring and let $(P,Q,\psi)$ be a unital $R$-system that satisfies Condition (FS'). Let $\epsilon_i$ be the idempotents defined above. The following assertions hold for every $i \geq 0$:
\begin{enumerate}[(a)]
\begin{item}
For any $s \in \mathcal{T}_i$ we have that $\epsilon_i s = s$;
\end{item}
\begin{item}
For any $t \in \mathcal{T}_{-i}$ we have that $t \epsilon_i =t$.
\end{item}
\end{enumerate}
Consequently, $\mathcal{T}_i \mathcal{T}_{-i}$ is a unital ideal with multiplicative identity element $\epsilon_i$ for every $i \geq 0$. 
\label{prop:positive_ideals}
\end{proposition}
\begin{proof}
Note that $\mathcal{T}_0$ is a unital ring with multiplicative identity element $\epsilon_0 = \iota_R(1_R)$. The statements are clear for $i = 0$. 

(a): Take an arbitrary positive integer $i$. Consider a monomial $\iota_{Q^{\otimes m}}(q) \iota_{P^{\otimes n}}(p)$ where $m,n$ are non-negative integers such that $m-n=i$. Then, $0 < i \leq m$. By Lemma \ref{lem:idem_seq}, $\epsilon_m \geq \epsilon_i$. Hence, 
\begin{align*}
\iota_{Q^{\otimes m}}(q) \iota_{P^{\otimes n}}(p) &= \iota_{Q^{\otimes m}}(\Delta^{m}(1_R)(q)) \iota_{P^{\otimes n}}(p) = \pi(\Delta^m(1_R)) \iota_{Q^{\otimes m}}(q) \iota_{P^{\otimes n}}(p) \\ &= \epsilon_m \iota_{Q^{\otimes m}}(q) \iota_{P^{\otimes n}}(p) = \epsilon_i \epsilon_m \iota_{Q^{\otimes m}}(q) \iota_{P^{\otimes n}}(p) = \epsilon_i \iota_{Q^{\otimes m}}(q) \iota_{P^{\otimes n}}(p).
\end{align*} Any element $s \in \mathcal{T}_i$ is a finite sum of elements of the above form (see (\ref{eq:grading})). Hence, it follows that $\epsilon_i s = s$. 

(b): Take an arbitrary positive integer $i$. Consider a monomial $\iota_{Q^{\otimes m}}(q) \iota_{P^{\otimes n}}(p)$ where $m,n$ are non-negative integers such that $m-n=-i$. Then $0 < i \leq n$. By Lemma \ref{lem:idem_seq}, $\epsilon_n \geq \epsilon_i$. Hence, $\iota_{Q^{\otimes m}}(q) \iota_{P^{\otimes n}}(p) = \iota_{Q^{\otimes m}}(q) \iota_{P^{\otimes n}}(\Gamma^n(1_R)(p)) = \iota_{Q^{\otimes m}}(q) \iota_{P^{\otimes n}}(p) \chi(\Gamma^n(1_R)) = \iota_{Q^{\otimes m}}(q) \iota_{P^{\otimes n}}(p) \epsilon_n = \iota_{Q^{\otimes m}}(q) \iota_{P^{\otimes n}}(p) \epsilon_n \epsilon_i =  \iota_{Q^{\otimes m}}(q) \iota_{P^{\otimes n}}(p)\epsilon_i$. Since any element $t \in \mathcal{T}_{-i}$ is a finite sum of elements of the above form, it follows that $ t \epsilon_i = t$. 
\end{proof}

We will see that restricting our attention to semi-full covariant representations $(S,T,\sigma,B)$ makes life easier. This special type of graded covariant representations have the property that the image of $\psi_k$ is enough to generate the ideal $B_{-k} B_k$ for $k \geq 0$ (see Definition \ref{def:semi-full}).  We first prove that the property of being semi-full is invariant under isomorphism in the category of surjective covariant representations $\mathcal{C}_{(P,Q,\psi)}$.




\begin{proposition}
Let $R$ be a ring, let $(P,Q,\psi)$ be an $R$-system and suppose that $(S,T,\sigma,B) \cong_{\text{r}} (S',T',\sigma', B')$ are two isomorphic covariant representations of $(P,Q,\psi)$. If $(S,T,\sigma,B)$ is semi-full, then $(S',T',\sigma', B')$ is semi-full. 
\label{prop:semi-full_iso}
\end{proposition}
\begin{proof}
Let $\phi \colon B \to B'$ be the $\mathbb{Z}$-graded isomorphism coming from Lemma \ref{lem:rep_maps}. Hence, \begin{align*}
B_{-k}' B_k' &= \phi(B_{-k}) \phi(B_{k}) = \phi(B_{-k} B_k) = \phi(I_{\psi, \sigma}^{(k)}) \\ &= (\{ \phi \circ \sigma(\psi_k(p \otimes q)) \mid  p \in P^{\otimes k}, q \in Q^{\otimes k} \}) = I_{\psi,\sigma'}^{(k)}.
\end{align*}
Thus, $(S',T', \sigma', B')$ is semi-full.
\end{proof}

We now establish sufficient conditions for a semi-full covariant representation to be nearly epsilon-strongly $\mathbb{Z}$-graded.

\begin{proposition}
Let $R$ be an s-unital ring and let $(P,Q,\psi)$ be an s-unital $R$-system. Suppose that $(S,T,\sigma, B)$ is a semi-full covariant representation of $(P,Q,\psi)$ and that the following assertions hold:
\begin{enumerate}[(a)]
\begin{item}
$(P,Q,\psi)$ satisfies Condition (FS),
\end{item}
\begin{item}
$I_{\psi,\sigma}^{(k)}$ is s-unital for $k \geq 0$.
\end{item}
\end{enumerate}
Then, $B$ is nearly epsilon-strongly $\mathbb{Z}$-graded.
\label{prop:nearly_epsilon_suff}
\end{proposition}
\begin{proof}
Let $\mathcal{T}_{(P,Q,\psi)} = \bigoplus_{i \in \mathbb{Z}} \mathcal{T}_i$ be the Toeplitz ring associated to the $R$-system $(P,Q,\psi)$. By Proposition \ref{prop:nearly_toeplitz}(c), $\mathcal{T}_i \mathcal{T}_{-i}$ is s-unital for every $ i \geq 0$. By Theorem \ref{thm:universal}, there is a $\mathbb{Z}$-graded ring epimorphism $\eta \colon \mathcal{T}_{(P,Q,\psi)} \to B$. Since the image of an s-unital ring under a ring homomorphism is in turn s-unital, it follows that $B_i B_{-i} = \eta(\mathcal{T}_i) \eta(\mathcal{T}_{-i})=\eta(\mathcal{T}_i \mathcal{T}_{-i})$ is s-unital for every $i \geq 0$. Furthermore, by Proposition \ref{prop:nearly_toeplitz}(d), we have that $\mathcal{T}_i = \mathcal{T}_i \mathcal{T}_{-i} \mathcal{T}_i$ for every $i \in \mathbb{Z}$. Applying $\eta$ to both sides yields, $B_i = B_i B_{-i} B_i$. Hence, $B$ is symmetrically $\mathbb{Z}$-graded.

Next, we show that $B_i B_{-i}$ is s-unital for $i < 0$. Since $(S,T,\sigma, B)$ is semi-full, we have that $B_{-k} B_k = I_{\psi,\sigma}^{(k)}$ for $k \geq 0$. Hence, (b) implies that $B_i B_{-i}$ is s-unital for $i < 0$. Thus, we have showed that $B_i B_{-i}$ is s-unital for $i \in \mathbb{Z}$ and that $B$ is symmetrically $\mathbb{Z}$-graded. By Proposition \ref{prop:nearly_char}(a), it follows that  $B=\bigoplus_{i \in \mathbb{Z}} B_i$ is nearly epsilon-strongly $\mathbb{Z}$-graded.
\end{proof}

The proof of the following proposition is entirely analogous to the proof of Proposition \ref{prop:nearly_epsilon_suff}.

\begin{proposition}
Let $R$ be a unital ring and let $(P,Q,\psi)$ be a unital $R$-system. Suppose that $(S,T,\sigma, B)$ is a semi-full covariant representation of $(P,Q,\psi)$ and that the following assertions hold:
\begin{enumerate}[(a)]
\begin{item}
$(P,Q,\psi)$ satisfies Condition (FS'),
\end{item}
\begin{item}
$I_{\psi,\sigma}^{(k)}$ is unital for $k \geq 0$.
\end{item}
\end{enumerate}
Then, $B$ is epsilon-strongly $\mathbb{Z}$-graded.
\label{prop:epsilon_suff}
\end{proposition}

On the other hand, a covariant representation $(S,T,\sigma, B)$ does not need to be semi-full for the ring $B$ to be epsilon-strongly $\mathbb{Z}$-graded (see Example \ref{ex:1}).

\section{Characterization up to graded isomorphism}
\label{sec:characterization}

In this section, we finally give characterizations of unital strongly, nearly epsilon-strongly and epsilon-strongly $\mathbb{Z}$-graded  Cuntz-Pimsner rings up to $\mathbb{Z}$-graded isomorphism.

\begin{theorem}
Let $\mathcal{O}_{(P,Q,\psi)}$ be a Cuntz-Pimsner ring of some system $(P,Q,\psi)$. If $\mathcal{O}_{(P,Q,\psi)}$ is nearly epsilon-strongly $\mathbb{Z}$-graded and $ \Ann_{\mathcal{O}_0}(\mathcal{O}_1) \cap (\Ann_{\mathcal{O}_0}(\mathcal{O}_1))^\bot = \{ 0 \}$, then,
$$ \mathcal{O}_{(P,Q,\psi)} \cong_{\text{gr}} \mathcal{O}_{(P', Q', \psi')},$$
where $(P',Q',\psi')$ is an $R'$-system such that $\mathcal{O}_{(P',Q',\psi')}$ is well-defined and the following assertions hold: 
\begin{enumerate}[(a)]
\begin{item}
$(P',Q',\psi')$ is an s-unital $R'$-system;
\end{item}
\begin{item}
$(\iota_{P'}^{CP}, \iota_{Q'}^{CP}, \iota_{R'}^{CP}, \mathcal{O}_{(P',Q',\psi')})$ is a semi-full covariant representation of $(P',Q',\psi')$;
\end{item}
\begin{item}
$(P', Q', \psi')$ satisfies Condition (FS);
\end{item}
\begin{item}
$I_{\psi', \iota_{\mathcal{O}_0}^{CP}}^{(k)}$ is s-unital for $k \geq 0$.
\end{item}
\end{enumerate}
Conversely, if $(P',Q',\psi')$ is an $R'$-system such that $\mathcal{O}_{(P',Q',\psi')}$ is well-defined and (a)-(d) hold, then $\mathcal{O}_{(P', Q', \psi')}$ is nearly epsilon-strongly $\mathbb{Z}$-graded.
\label{thm:1}
\end{theorem}
\begin{proof}
If the Cuntz-Pimsner ring $\mathcal{O}_{(P,Q,\psi)}$ is nearly epsilon-strongly $\mathbb{Z}$-graded and the condition $ \Ann_{\mathcal{O}_0}(\mathcal{O}_1) \cap (\Ann_{\mathcal{O}_0}(\mathcal{O}_1))^\bot = \{ 0 \}$ holds, then it follows from Corollary \ref{cor:reduction} that the Cuntz-Pimsner ring is graded isomorphic to $\mathcal{O}_{(\mathcal{O}_{-1}, \mathcal{O}_1,\psi')}$ and that (a)-(d) are satisfied. 

Conversely, let $(P', Q', \psi')$ be an $R'$-system such that $\mathcal{O}_{(P',Q',\psi')}$ exists and (a)-(d) are satisfied. Applying Proposition \ref{prop:nearly_epsilon_suff} to the covariant representation $(\iota_{P'}^{CP}, \iota_{Q'}^{CP}, \iota_{R'}^{CP}, \mathcal{O}_{(P',Q',\psi')})$, it follows that $\mathcal{O}_{(P',Q',\psi')}$ is nearly epsilon-strongly $\mathbb{Z}$-graded.
\end{proof}

For epsilon-strongly $\mathbb{Z}$-graded Cuntz-Pimsner rings, we obtain the following result:

\begin{theorem}
Let $\mathcal{O}_{(P,Q,\psi)}$ be a Cuntz-Pimsner ring of some system $(P,Q,\psi)$. If $\mathcal{O}_{(P,Q,\psi)}$ is epsilon-strongly $\mathbb{Z}$-graded and $ \Ann_{\mathcal{O}_0}(\mathcal{O}_1) \cap (\Ann_{\mathcal{O}_0}(\mathcal{O}_1))^\bot = \{ 0 \}$, then, 
$$ \mathcal{O}_{(P,Q,\psi)} \cong_{\text{gr}} \mathcal{O}_{(P', Q', \psi')},$$
where $(P',Q',\psi')$ is an $R'$-system such that $\mathcal{O}_{(P',Q',\psi')}$ is well-defined and the following assertions hold:
\begin{enumerate}[(a)]
\begin{item}
$(P',Q',\psi')$ is a unital $R'$-system;
\end{item}
\begin{item}
$(\iota_{P'}^{CP}, \iota_{Q'}^{CP}, \iota_{R'}^{CP}, \mathcal{O}_{(P',Q',\psi')})$ is a semi-full covariant representation of $(P',Q',\psi')$;
\end{item}
\begin{item}
$(P', Q', \psi')$ satisfies Condition (FS');
\end{item}
\begin{item}
$I_{\psi',\iota_{\mathcal{O}_0}^{CP}}^{(k)}$ is unital for $k \geq 0$.
\end{item}
\end{enumerate}
\label{thm:epsilon}
Conversely, if $(P',Q',\psi')$ is an $R'$-system such that $\mathcal{O}_{(P',Q',\psi')}$ is well-defined and (a)-(d) hold, then $\mathcal{O}_{(P', Q', \psi')}$ is epsilon-strongly $\mathbb{Z}$-graded.
\end{theorem}
\begin{proof}
Assume that $(P',Q',\psi')$ is an $R'$-system such that $\mathcal{O}_{(P',Q',\psi')}$ exists and the assertions in (a)-(d) hold. Then Proposition \ref{prop:epsilon_suff} implies that $\mathcal{O}_{(P',Q',\psi')}$ is epsilon-strongly $\mathbb{Z}$-graded.

Conversely, assume that $\mathcal{O}_{(P,Q,\psi)}$ is epsilon-strongly $\mathbb{Z}$-graded and $ \Ann_{\mathcal{O}_0}(\mathcal{O}_1) \cap (\Ann_{\mathcal{O}_0}(\mathcal{O}_1))^\bot = \{ 0 \}$. Note that, in particular, $\mathcal{O}_{(P,Q,\psi)}$ is nearly epsilon-strongly $\mathbb{Z}$-graded. Hence, by Theorem \ref{thm:1}, $\mathcal{O}_{(P,Q,\psi)} \cong_{\text{gr}} \mathcal{O}_{(\mathcal{O}_{-1}, \mathcal{O}_1, \psi')}$ where $(\mathcal{O}_{-1}, \mathcal{O}_1, \psi')$ is an s-unital $\mathcal{O}_0$-system that satisfies Condition (FS) and such that (b) is satisfied. Furthermore (see Corollary \ref{cor:reduction}), 
\begin{equation}
(i_{\mathcal{O}_{-1}}, i_{\mathcal{O}_1}, i_{\mathcal{O}_0}, \mathcal{O}_{(P,Q,\psi)}) \cong_{\text{r}} (\iota_{\mathcal{O}_{-1}}^{CP}, \iota_{\mathcal{O}_1}^{CP}, \iota_{\mathcal{O}_0}^{CP}, \mathcal{O}_{(\mathcal{O}_{-1}, \mathcal{O}_1, \psi')}).
\label{eq:6}
\end{equation} 

First note that since the $\mathbb{Z}$-grading is assumed to be epsilon-strong it follows that $\mathcal{O}_i$ is a unital $\mathcal{O}_i \mathcal{O}_{-i} \text{--} \mathcal{O}_{-i} \mathcal{O}_i$-bimodule for each $i \in \mathbb{Z}$ (see Definition \ref{def:nystedt_epsilon}). This implies that $(\mathcal{O}_{-1}, \mathcal{O}_1, \psi')$ is a unital $\mathcal{O}_0$-system. Hence, (a) is satisfied. 

Next, we prove that the $\mathcal{O}_0$-system $(\mathcal{O}_{-1}, \mathcal{O}_1, \psi')$ satisfies Condition (FS'). Since $\mathcal{O}_{(P,Q,\psi)}$ is assumed to be epsilon-strongly $\mathbb{Z}$-graded, it follows from \cite[Prop. 7(iv)]{nystedt2016epsilon} that $\mathcal{O}_i$ is a finitely generated $\mathcal{O}_0$-bimodule for every $i \in \mathbb{Z}$. In particular, $\mathcal{O}_1$ and $\mathcal{O}_{-1}$ are finitely generated $\mathcal{O}_0$-bimodules and it follows from Proposition \ref{prop:fsprime_char}(c) that $(\mathcal{O}_{-1}, \mathcal{O}_1, \psi)$ satisfies Condition (FS'). In other words, (c) holds.

%

Moreover, it follows from Proposition \ref{prop:nearly_char}(b) that, in particular, $\mathcal{O}_{-k} \mathcal{O}_k$ is unital for $k \geq 0$. Hence, $\mathcal{O}_{-k} \mathcal{O}_k = I_{\psi', \iota_{\mathcal{O}_0}^{CP}}^{(k)}$ is unital for $k \geq 0$. This establishes (d). 
\end{proof}


For unital strongly $\mathbb{Z}$-graded Cuntz-Pimsner rings, we obtain the following complete characterization:

\begin{theorem}
Let $\mathcal{O}_{(P,Q,\psi)}$ be a Cuntz-Pimsner ring of some system $(P,Q,\psi)$. Then, $\mathcal{O}_{(P,Q,\psi)}$ is unital strongly $\mathbb{Z}$-graded if and only if 
$$ \mathcal{O}_{(P,Q,\psi)} \cong_{\text{gr}} \mathcal{O}_{(P', Q', \psi')}$$ 
where $(P',Q',\psi')$ is an $R'$-system such that $\mathcal{O}_{(P',Q',\psi')}$ is well-defined and the following assertions hold:
\begin{enumerate}[(a)]
\begin{item}
$(P',Q',\psi')$ is a unital $R'$-system;
\end{item}
\begin{item}
$(\iota_{P'}^{CP}, \iota_{Q'}^{CP}, \iota_{R'}^{CP}, \mathcal{O}_{(P',Q',\psi')})$ is a semi-full and faithful covariant representation of $(P',Q',\psi')$;
\end{item}
\begin{item}
$\psi'$ is surjective.
\end{item}
\end{enumerate}
\label{thm:2}
\end{theorem}
\begin{proof}
By Proposition \ref{prop:strong_suff}, (a) and (c) are sufficient for the ring $\mathcal{O}_{(P',Q',\psi')}$ to be strongly $\mathbb{Z}$-graded. 

Conversely, assume that $\mathcal{O}_{(P,Q,\psi)}$ is unital strongly $\mathbb{Z}$-graded. In particular, $\mathcal{O}_{(P,Q,\psi)}$ is epsilon-strongly $\mathbb{Z}$-graded. Moreover, $ \Ann_{\mathcal{O}_0}(\mathcal{O}_1) \cap (\Ann_{\mathcal{O}_0}(\mathcal{O}_1))^\bot = \{ 0 \}$ by Lemma \ref{lem:semi-prime}(b). Then, by Theorem \ref{thm:epsilon}, $\mathcal{O}_{(P,Q,\psi)} \cong_{\text{gr}} \mathcal{O}_{(\mathcal{O}_{-1}, \mathcal{O}_1, \psi')}$ where $(\mathcal{O}_{-1}, \mathcal{O}_1, \psi')$ satisfies Condition (FS'), (b) is satisfied and $I_{\psi',\iota_{\mathcal{O}_0}^{CP}}^{(k)}$ is unital for $k \geq 0$. Since $\mathcal{O}_{(\mathcal{O}_{-1}, \mathcal{O}_1, \psi)}$ is unital strongly $\mathbb{Z}$-graded, $$1_{\mathcal{O}_{(\mathcal{O}_{-1}, \mathcal{O}_1, \psi')}} = \iota_{\mathcal{O}_0}^{CP}(1_{\mathcal{O}_0}) \in \mathcal{O}_0 = \mathcal{O}_{-1} \mathcal{O}_1 = I_{\psi', \iota_{\mathcal{O}_0}^{CP}}^{(1)}.$$ Since $\iota_{\mathcal{O}_0}^{CP}$ is injective, we get that $1_{\mathcal{O}_0} \in \text{Im}(\psi')$. Hence, $\psi'$ is surjective.

Furthermore, since $\mathcal{O}_{(\mathcal{O}_{-1}, \mathcal{O}_1, \psi')}$ is an epsilon-strongly $\mathbb{Z}$-graded ring that is also strongly $\mathbb{Z}$-graded, we must have $\epsilon_1 = 1$ (see \cite[Prop. 8]{nystedt2016epsilon}) where $\epsilon_1$ is the multiplicative identity element of the ring $\mathcal{O}_1 \mathcal{O}_{-1}$. By Condition (FS') and Proposition \ref{prop:fsprime_char}(b), we have that $\Delta(1_{\mathcal{O}}) \in \mathcal{F}_P(Q)$. Then, by Proposition \ref{prop:pi}, $\pi_{\iota_{\mathcal{O}_{1}}^{CP}, \iota_{\mathcal{O}_{-1}}^{CP}}(\Delta(1_{\mathcal{O}_0})) \in \mathcal{O}_1 \mathcal{O}_{-1}$ is a multiplicative identity element of $\mathcal{O}_1 \mathcal{O}_{-1}$. Thus, $\pi_{\iota_{\mathcal{O}_{1}}^{CP}, \iota_{\mathcal{O}_{-1}}^{CP}}(\Delta(1_{\mathcal{O}_0}))= \epsilon_1 = 1$ and therefore $(\iota_{\mathcal{O}_{-1}}^{CP}, \iota_{\mathcal{O}_{1}}^{CP}, \iota_{\mathcal{O}_0}^{CP}, \mathcal{O}_{(\mathcal{O}_{-1}, \mathcal{O}_1, \psi')})$ is a faithful representation of $(\mathcal{O}_{-1}, \mathcal{O}_1, \psi')$.
\end{proof}

\section{Examples}
\label{sec:ex}

In this section, we collect some important examples.

\begin{example}(Non-nearly epsilon-strongly $\mathbb{Z}$-graded Cuntz-Pimsner ring)
Let $R$ be an idempotent ring that is not s-unital (see e.g. \cite[Expl. 5]{nystedt2018unital}). Put $P=Q=\{ 0 \}$ and let $\psi \equiv 0$ be the zero map. Note that $(P,Q,\psi)$ is an $R$-system that satisfies Condition (FS') trivially. It is not hard to see that the Toeplitz ring is given by $\mathcal{T}_0 = R$, and $\mathcal{T}_{i} = \{ 0 \}$ for all $i \ne 0$. Furthermore, note that $\ker \Delta = R$. Recall that an ideal $J$ of $R$ is called faithful if $J \cap \ker \Delta = \{ 0 \}$. Clearly, $J := (0)$ is the maximal faithful and $\psi$-compatible ideal of $R$. It follows that the Cuntz-Pimsner ring $\mathcal{O}_{(P,Q,\psi)}$ is well-defined and coincides with the Toeplitz ring. Since $\mathcal{T}_0 = R = R^2 = \mathcal{T}_0 \mathcal{T}_0$ is not s-unital it follows by Proposition \ref{prop:nearly_char}(a) that the Cuntz-Pimsner ring $\mathcal{O}_{(P,Q,\psi)}=\mathcal{T}_{(P,Q,\psi)}$ is not nearly epsilon-strongly $\mathbb{Z}$-graded. This shows that the assumption of $(P,Q,\psi)$ being an s-unital system in Proposition \ref{prop:nearly_epsilon_suff} cannot be removed. 
\label{ex:2}
\end{example}

The following example shows that for some graphs, the standard Leavitt path algebra covariant representation is semi-full (see Section \ref{sec:lpa}).

\begin{example}
\begin{displaymath}
	\xymatrix{
	\bullet_{v_1}  \ar[r]^{f_1} & \bullet_{v_2}  
	}
\end{displaymath}
Let $K$ be a unital ring and let $E$ consist of two vertices $v_1, v_2$ connected by a single edge $f$. Consider the associated standard Leavitt path algebra system $(P,Q,\psi)$ and the standard Leavitt path algebra covariant representation $(\iota_Q^{CP}, \iota_P^{CP}, \iota_R^{CP}, \mathcal{O}_{(P,Q,\psi)})$. To save space we write $I_k = I_{\psi, \iota_R^{CP}}^{(k)}$ for $k \geq  0$. Note that $I_0 = (\{ v_1, v_2 \})$, $ I_1 = ( v_2 ) $ and $ I_k = (0)$ for $k > 2$. 
Furthermore, since $f_1 f_1^* = v_1$ we see that $(L_K(E))_0 = I_0$. 

Moreover, note that $(L_K(E))_1 = \Span_K \{ f_1 \},$ $ (L_K(E))_{-1} = \Span_K \{ f_1^* \}$ and hence we see that $(L_K(E))_{-1} (L_K(E))_1 = (v_2) = I_1$. Thus, $(\iota_P, \iota_Q, \iota_R, \mathcal{O}_{(P,Q,\psi)})$ is a semi-full covariant representation of $(P,Q,\psi)$. Furthermore, $(P,Q,\psi)$ satisfies Condition (FS') since $E$ is finite (see Lemma \ref{lem:finitely_many_edges}) and $I_k$ is unital for $k \geq 0$. Thus, $L_K(E)$ is epsilon-strongly $\mathbb{Z}$-graded by Theorem \ref{thm:epsilon}.

\label{ex:3}
\end{example}

In general, however, it is not true that the standard Leavitt path algebra covariant representation is semi-full as the following example shows.

\begin{example}(cf. \cite[Expl. 4.1]{nystedt2017epsilon})
Let $K$ be a unital ring and consider the following finite directed graph $E$.
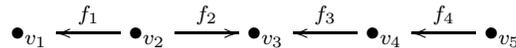
\begin{displaymath}
	\xymatrix{
	\bullet_{v_1} & \ar[l]_{f_1} \bullet_{v_2}  \ar[r]^{f_2} & \bullet_{v_3} & \ar[l]_{f_3} \bullet_{v_4} & \ar[l]_{f_4} \bullet_{v_5}
	}
\end{displaymath}
\noindent
Let $(P,Q,\psi)$ be the standard Leavitt path algebra system associated to $E$ and consider the standard Leavitt path algebra covariant representation, 
\begin{equation}
(\iota_Q^{CP}, \iota_P^{CP}, \iota_R^{CP}, \mathcal{O}_{(P,Q,\psi)}).
\label{eq:3}
\end{equation}
We write $S_i= (L_K(E))_i$ and $I_i = I_{\psi, \iota_R^{CP}}^{(i)}$ to save space. Note that,
\begin{align*}
	&S_0 = \Span_K \{ v_1, v_2, v_3, v_4, v_5, f_1f_1^*, f_2f_3^*, f_2f_2^*, f_3f_2^* \}, \\
	&S_1 = \Span_K \{ f_1, f_2, f_3, f_4, f_4f_3f_2^* \}, \quad \quad
	S_{-1} = \Span_K \{ f_1^*, f_2^*, f_3^*, f_4^*, f_2f_3^*f_4^* \}, \\
	&S_2 = \Span_K \{ f_4f_3 \}, \quad \quad S_{-2} = \Span_K \{ f_3^*f_4^* \}, \quad \text{and} \quad S_{n} = \{0\}, \text{ for } |n|>2.
\end{align*}

\noindent
Furthermore,
\begin{align*}
	&I_0 = ( \{ v_1, v_2, v_3, v_4, v_5 \}), 
	&I_1 = ( \{ v_1, v_3, v_4 \}),  \\
	&I_2 = ( \{  v_3 \} ), 
	&I_k = ( 0 ), \quad k > 2.
\end{align*}

In particular, we have that $S_{-1} S_1 = (\{ v_1, v_3, v_4, f_2 f_2^* \}) \supsetneqq I_1$ because $f_2 f_2^* \not \in I_1$. Hence, the standard Leavitt path algebra covariant representation is not semi-full. In any case, however, we have that $\mathcal{O}_{(P,Q,\psi)} \cong_{\text{gr}} L_K(E)$ (see Section \ref{sec:lpa}). On the other hand, by Proposition \ref{prop:lpa-cp}, we have that  $L_K(E)$ is pre-CP. Thus, by Corollary \ref{cor:nearly_cuntz}, $L_K(E)$ is realized by the Cuntz-Pimsner representation,
\begin{equation}
(\iota_{(L_K(E))_{-1}}^{CP}, \iota_{(L_K(E))_1}^{CP}, \iota_{(L_K(E))_0}^{CP}, \mathcal{O}_{((L_K(E))_{-1}, (L_K(E))_{1}, \psi')})
\label{eq:4}
\end{equation} of the $(L_K(E))_0$-system $(L_K(E))_{-1}, (L_K(E))_1, \psi')$. Moreover, the corollary implies that (\ref{eq:4}) is semi-full and $\mathcal{O}_{((L_K(E))_{-1}, (L_K(E))_{1}, \psi')} \cong_{\text{gr}} L_K(E)$. Since (\ref{eq:3}) is not semi-full and (\ref{eq:4}) is semi-full, it follows by Proposition \ref{prop:semi-full_iso} that the covariant representations (\ref{eq:3}) and (\ref{eq:4}) cannot be isomorphic. Thus, $L_K(E)$ is realizable as a Cuntz-Pimsner ring in two different ways. 
\label{ex:1}
\end{example}

The following example shows that (a) is crucial in Theorem \ref{thm:epsilon}. It also gives an example of a nearly epsilon-strongly $\mathbb{Z}$-graded ring that is not epsilon-strongly $\mathbb{Z}$-graded.
\begin{example}(cf. \cite[Expl. 4.5]{lannstrom2018induced})
Let $K$ be a unital ring and consider the infinite discrete graph $E$ consisting of countably infinitely many vertices but no edges. 

\begin{displaymath}
	\xymatrix{
	\bullet_{v_1} \qquad \bullet_{v_2} \qquad \bullet_{v_3} \qquad  \bullet_{v_4}  \qquad \bullet_{v_5}  \qquad \bullet_{v_6}  \qquad \bullet_{v_7}  \qquad \bullet_{v_8}   \qquad \bullet_{v_9}   \qquad \bullet_{v_{10}}   \qquad \dots
	}
\end{displaymath}
The standard Leavitt path algebra system is given by $R=\bigoplus_{v \in E^0} \eta_{v}$, $P=Q=\{ 0 \}$. The $R$-system $(P,Q,\psi)$ trivially satisfies Condition (FS'). However, $(P,Q,\psi)$ is not unital as $R$ does not have a multiplicative identity element. However, note that $(P,Q,\psi)$ is s-unital. 

We show that the standard  Leavitt path algebra covariant representation of $E$ is semi-full. Since $P=Q=\{0 \}$ and $\psi =0$ it follows that the grading is given by $\mathcal{O}_0 = R$ and $\mathcal{O}_i = \{ 0 \}$ for $i \ne 0$ (see Example \ref{ex:2}). Furthermore, $I_{\psi, \iota_R^{CP}}^{(k)} = ( 0 )$ for $k > 0$. Thus, the standard Leavitt path algebra covariant representation satisfies (b)-(d) in Theorem \ref{thm:epsilon} but not (a). Since $E$ contains infinitely many vertices, $L_K(E)$ is not unital (see \cite[Lem. 1.2.12]{abrams2017leavitt}). By Remark \ref{rem:unital_epsilon}, $L_K(E)$ is not epsilon-strongly $\mathbb{Z}$-graded (cf.  \cite[Expl. 4.5]{lannstrom2018induced}). Thus, (a) in Theorem \ref{thm:1} cannot be removed. On the other hand, it follows from Theorem \ref{thm:1} that $L_K(E)$ is nearly epsilon-strongly $\mathbb{Z}$-graded. 
\end{example}

\section{Noetherian and artinian corner skew Laurent polynomial rings}
\label{sec:app}

We end this article by characterizing noetherian and artinian corner skew Laurent polynomial rings. The following proposition can be proved in a straightforward manner using direct methods, but we show it as a special case of our results.

\begin{proposition}
Let $R$ be a unital ring, let $e \in R$ be an idempotent and let $\alpha \colon R \to eRe$ be a corner ring isomorphism. Then the corner skew Laurent polynomial ring $R[t_{+}, t_{-}; \alpha]$ is epsilon-strongly $\mathbb{Z}$-graded.
\label{prop:corner_skew_epsilon}
\end{proposition}
\begin{proof}
Recall that $R[t_{+}, t_{-}; \alpha] = \bigoplus_{i \in \mathbb{Z}} A_i$ is $\mathbb{Z}$-graded by putting $A_0 = R $, $A_i = R t_{+}^i$ for $i < 0$  and $A_i = t_{-}^i R$ for $i > 0$. Let  $\psi' \colon A_{-1} \otimes A_1 \to A_0$ be the map defined by $\psi(a' \otimes a) = a' a$ for $a' \in A_{-1}$ and $a \in A_1$. Since $A_0=R$ is a unital ring, the $A_0$-system $(A_{-1},A_1, \psi')$ is unital. In \cite[Expl. 3.4]{2018arXiv180810114O}, it is shown that $R[t_{+}, t_{-}; \alpha]$ satisfies the conditions in Theorem \ref{thm:clark}. This implies that $(A_{-1}, A_1, \psi')$ satisfies Condition (FS) and, 
\begin{equation}
 (i_{A_{-1}}, i_{A_1}, i_{A_{0}}, R[t_{+}, t_{-}; \alpha]) \cong_{\text{r}} (\iota_{A_{-1}}^{CP}, \iota_{A_1}^{CP}, \iota_{A_0}^{CP}, \mathcal{O}_{(A_{-1}, A_1, \psi')}).
 \label{eq:7}
\end{equation}
Note that $A_1 = t_{-} R$ is finitely generated as a right $A_0$-module and $A_{-1} = R t_{+}$ is finitely generated as a left $A_0$-module. It follows from Proposition \ref{prop:fsprime_char}, that $(A_{-1}, A_1, \psi')$ satisfies Condition (FS'). Furthermore, by Proposition \ref{rem:5}, the covariant representation (\ref{eq:7}) is semi-full.

Next, we show that $I_{\psi', \iota_{R}^{CP}}^{(k)} = A_{-k} A_k$ is unital with multiplicative identity element $t_{+}^k t_{-i}^k = i(e) \in A_{-k} A_k$ for each $k > 0$. Fix a non-negative integer $k > 0$ and note that any  element $x \in A_{-k}A_k=(R t_{+}^k)( t_{-i}^k R)$ is a finite sum of elements of the form $r t_{+}^k t_{-}^k r' = t_{+}^k \alpha^{-k}(r) t_{-}^k r' = r t_{+}^k \alpha^{-k}(r') t_{-}^k$ where $r, r' \in R$. For any $r,r'\in R$, we get that, $$i(e) r t_{+}^k t_{-}^k r' =  (t_{+}^k t_{-}^k) (t_{+}^k \alpha^{-k}(r) t_{-}^k r') = t_{+}^k (t_{-}^k t_{+}^k) \alpha^{-k}(r) t_{-}^k r' = t_{+}^k (1) \alpha^{-k}(r) t_{-}^k r' = r t_{+}^k t_{-}^k r'.$$ It follows that $i(e) x = x$. A similar argument shows that $x i(e) = x$. By Theorem \ref{thm:epsilon}, it now follows that $R[t_{+}, t_{-}; \alpha]$ is epsilon-strongly $\mathbb{Z}$-graded.
\end{proof}


We recall the following Hilbert basis theorem for epsilon-strongly $\mathbb{Z}$-graded rings.

\begin{theorem}(\cite[Thm. 1.1, Thm. 1.2]{lannstrom2018chain})
Let $S=\bigoplus_{i \in \mathbb{Z}} S_i$ be an epsilon-strongly $\mathbb{Z}$-graded ring. The following assertions hold:
\begin{enumerate}[(a)]
\begin{item}
 If $S_0$ is left (right) noetherian, then $S$ is left (right) noetherian;
\end{item}
\begin{item}
If $S_0$ is left (right) artinian and there exists some positive integer $n$ such that $S_i = \{ 0 \}$ for all $|i| > n$, then $S$ is left (right) artinian. 
\end{item}
\end{enumerate}
\label{thm:hilbert}
\end{theorem}

Applying Theorem \ref{thm:hilbert} to the special case of corner skew Laurent polynomial rings, we obtain the following result. 

\begin{corollary}
Let $R$ be a unital ring and let $\alpha \colon R \to eRe$ be a ring isomorphism where $e$ is an idempotent of $R$. Consider the corner skew Laurent polynomial ring $R[t_{+}, t_{-}; \alpha]$. The following assertions hold:
\begin{enumerate}[(a)]
\begin{item}
$R[t_{+}, t_{-}; \alpha]$ is left (right) noetherian if and only if $R$ is left (right) noetherian;
\end{item}
\begin{item}
$R[t_{+}, t_{-}; \alpha]$ is neither left nor right artinian.
\end{item}
\end{enumerate}
\label{cor:artinian}
\end{corollary}
\begin{proof}
(a): Straightforward.

(b): By Proposition \ref{prop:corner_skew_epsilon}, $R[t_{+}, t_{-}; \alpha] = \bigoplus_{i \in \mathbb{Z}} A_i$ is epsilon-strongly $\mathbb{Z}$-graded where $A_0 = R $, $A_i = R t_{+}^i$ for $i < 0$  and $A_i = t_{-}^i R$ for $i > 0$. By Theorem \ref{thm:hilbert}(b), $R[t_{+}, t_{-}; \alpha]$ is left (right) artinian if and only if $A_0$ is left (right) artinian and $|\text{Supp}(R[t_{+}, t_{-}; \alpha])| < \infty$. However, since $t_{+}^n \ne 0$ for every $n > 0$, it follows that $A_{-n} = R t_{+}^n \ne \{ 0 \}$ for every $n > 0$. Hence, $\text{Supp}(R[t_{+}, t_{-}; \alpha])$ is infinite and $R[t_{+}, t_{-}; \alpha]$ is neither left nor right artinian.
\end{proof}

\section*{acknowledgement}
This research was partially supported by the Crafoord Foundation (grant no. 20170843). The author is grateful to Eduard Ortega for pointing out the characterization in Proposition \ref{prop:fsprime_char}(c). The author is also grateful to Stefan Wagner, Johan Öinert and Patrik Nystedt for giving feedback and comments that helped to improve this manuscript. 


\printbibliography

\end{document}